\documentclass[a4paper,11pt]{article}
\usepackage[utf8]{inputenc}
\usepackage[T1]{fontenc}

\usepackage{indentfirst}
\usepackage{enumitem}
\usepackage{amsmath}
\usepackage{amssymb,amsmath,amsthm}
\usepackage{ebproof}
\usepackage{aliascnt}
\usepackage{hyperref}
\usepackage[capitalise]{cleveref}
\usepackage{ragged2e}
\usepackage{adjustbox}
\usepackage [autostyle, english = american]{csquotes}
\MakeOuterQuote{"}

\usepackage[backend=biber,style=alphabetic]{biblatex}
\usepackage[affil-it]{authblk}

\usepackage{tikz-cd}
\usetikzlibrary{decorations.markings,intersections}
\tikzset{shorten <>/.style={shorten >=#1,shorten <=#1}}
\tikzset{every picture/.prefix code=\DisableQuotes}

\tikzset{%
scalearrow/.style n args={3}{
  decoration={
    markings,
    mark=at position (1-#1)/2*\pgfdecoratedpathlength
      with {\coordinate (#2);},
    mark=at position (1+#1)/2*\pgfdecoratedpathlength
      with {\coordinate (#3);},
    },
  postaction=decorate,
  }
}

\usepackage{quiver}
\usepackage[mathcal]{euscript}
\usepackage{ stmaryrd }
\usepackage{ dsfont }
\usepackage{upgreek}
\usepackage{ mathrsfs }

\renewenvironment{abstract}
 {\small
  \begin{center}
  \bfseries \abstractname\vspace{-.5em}\vspace{0pt}
  \end{center}
  \list{}{%
    \setlength{\leftmargin}{2mm}
    \setlength{\rightmargin}{\leftmargin}%
  }%
  \item\relax}
 {\endlist}
 
 \let\oldtheorem\newtheorem
\RenewDocumentCommand{\newtheorem}{s m o m O{}}{%
\IfBooleanTF{#1}%
{\oldtheorem{#2}{#4}}%
{\IfNoValueTF{#3}{\oldtheorem{#2}{#4}[#5]}%
{\newaliascnt{#2}{#3}%
\oldtheorem{#2}[#2]{#4}%
\aliascntresetthe{#2}}}}

\newtheorem{theorem}{Theorem}[section]
\newtheorem{proposition}[theorem]{Proposition}
\newtheorem{lemma}[theorem]{Lemma}
\newtheorem{corollary}[theorem]{Corollary}

\theoremstyle{definition}
\newtheorem{definition}{Definition}[section]

\theoremstyle{remark}
\newtheorem{remark}{Remark}[section]

\title{Flat functors in the context of fibration categories}
\author{El Mehdi Cherradi}
\affil{IRIF - CNRS - Universit\'e Paris Cit\'e \\MINES ParisTech - Universit\'e PSL}
\date{}

\begin{document}

\maketitle

\begin{abstract}
 We investigate the connection between left exact $\infty$-functors between finitely complete quasicategories and exact functors between fibration categories, describing a procedure to approximate flat $\infty$-functors of the former type by exact functors of the latter type. As an application, we recover a proof of the DK-equivalence between the relative category of fibration categories and that of finitely complete quasicategories.
\end{abstract}

\tableofcontents

\newpage 

\section*{Introduction}
\addcontentsline{toc}{section}{Introduction}

While "general" models for $(\infty,1)$-categories ought to replicate the concepts of category theory in a homotopy-enabled context, they do not necessarily provide strict notions of composition and identity morphisms. As such, objects such as quasicategories do not usually come with a $1$-categorical structure. It is however convenient for many purposes to work present $(\infty,1)$-categories in term of $1$-categories with additional structure, such as model categories or fibration categories. When working with specific categories of $(\infty,1)$-categories, the process of replacing a quasicategory with a strict presentation thereof need not be restrictive. Specifically, there exist DK-equivalences:

\begin{enumerate}
 \item $$\mathbf{CMC} \to \mathbf{PrL}$$
from the relative category $\mathbf{CMC}$ of combinatorial model categories, left Quillen functors and left Quillen equivalence to the relative category $\mathbf{PrL}$ of presentable quasicategories, left adjoint $\infty$-functors and equivalences (this is established in \cite[Theorem 1.1]{pavlov2025combinatorial}).
\item $$\mathbf{FibCat} \to \mathbf{QCat}_{lex}$$ from the relative category of fibration categories, exact functors and DK-equivalences to the relative category of finitely complete quasicategories, left exact $\infty$-functors and equivalences (as proved in \cite[Theorem 4.9]{szumilo2014two}).
\end{enumerate}

In this document, our goal is to provide a new outlook on the correspondence between fibration categories and finitely complete quasicategories by means of the theory of flat functors and its $(\infty,1)$-categorical counterpart.

A functor between finitely complete categories is flat if and only if it preserves finite limits. While the definition of flat functors makes sense in a more general context (see Section 6.3 of \cite{borceux1994handbook}), we will only be considering flat functors between finitely complete categories, so that the notion coincides with that of finite limit-preserving functors. As a matter of fact, we will consider flat $\infty$-functors between $(\infty,1)$-categories presented by fibration categories. A subtle point is that fibration categories need not have all finite limits as $1$-categories, but the $(\infty,1)$-categories they present are indeed finitely complete. Likewise, exact functors present finite limit-preserving $\infty$-functors, namely, the flat functors we consider, but they need not be limit-preserving as $1$-functors.

Additionally, the framework of fibration categories provides a tool to compute finite limits in the corresponding $(\infty,1)$-category, especially pullbacks, by means of (special) $1$-categorical limits (e.g, pullbacks along fibrations). As such, the morphisms between fibration categories, that is, the exact functors, are not just presenting flat $\infty$-functors: they also preserve these special $1$-categorical limits, which encode the $\infty$-categorical ones.

One cannot expect exact functors to be flat in the $1$-categorical sense, as the preservation property they satisfy, by definition, only deals with particular pullbacks (pullbacks along fibrations). It is reasonable, nonetheless, to expect the theory of such functors to be richer, in some sense, than the theory of flat $\infty$-functors in general. Precisely, we are interested in the connection between flat functors and left Kan extensions. It is known that a functor $F : C \to E$, valued in some Grothendieck topos $E$, is (internally) flat precisely when its left Kan extension $\mathbf{Lan_y} : \mathbf{Set}^{C^{op}} \to E$ along the Yoneda embedding preserves finite limits (see \cite{maclane2012sheaves}, Corollary 3 in VII.9.1). In this section, we will study the left Kan extensions of exact functors. Explicitly, we will show in \cref{rlan_c} that it is possible to define such an extension, or rather an approximation of it, as a span of exact functor; hence reconciling the correct homotopy behavior with the rigidify of a $1$-categorical presentation that relies on fibrations and weak equivalences.

\section{Fibration categories and exact functors}

Recall the definition of a fibration category:

\begin{definition}
 A fibration category is a category $\mathcal{F}$ equipped with two classes of morphisms $\mathbf{W}$ (the weak equivalences) and $\mathbf{F}$ (the fibrations) that are stable under composition, and such that:
 
 \begin{itemize}
  \item $\mathcal{F}$ admits a terminal object $*$, and the unique map $x \to *$ is a fibration for every object $x$.
  \item $\mathcal{F}$ admits pullbacks along fibrations, and the base change of a fibration is a fibration.
  \item Trivial fibrations (fibrations that are also weak equivalences) are stable under pullback.
  \item The class of weak equivalence $\mathbf{W}$ satisfies the $2$-out-of-$3$ property.
  \item For every object $x$, there is a factorization of the diagonal $$x \to px \to x \times x$$ where $x \to px$ is a weak equivalence and $px \to x\times x$ is a fibration.  
 \end{itemize}

A functor $F : \mathcal{F} \to \mathcal{F}'$ between fibration categories is \textit{exact} when it preserves the corresponding structure: it maps fibrations (resp. weak equivalences) to fibrations (resp. weak equivalences), and preserves the terminal object as well as pullbacks along fibrations. 
\end{definition}

\begin{definition}
\label{kscpo_0}
 We define $\mathbf{Sp}_w$ to be the "homotopical span" category, that is, the following category
 $$\bullet \stackrel{\sim}{\leftarrow} \bullet \stackrel{\sim}{\rightarrow} \bullet$$
 where both maps are weak equivalences. $\mathbf{Sp}_w$ admits a Reedy category structure (which is an inverse one): the apex has degree $1$, and the two others objects have degree $0$.
\end{definition}
 
\begin{definition} 
 For $\mathcal{C}$ a fibration category, we write $P\mathcal{C}$ for the category of Reedy fibrant diagrams from $\mathbf{Sp}_w$ to $\mathcal{T}$.
\end{definition}

\begin{remark}
\label{span_to_product}
 By definition, the objects of  $P\mathcal{C}$ are Reedy fibrant spans $$x_0 \leftarrow X \rightarrow x_1$$ and the morphisms are natural transformations between such diagrams. Hence, the morphisms of $P\mathcal{C}$ correspond to diagrams:

\[\begin{tikzcd}[ampersand replacement=\&]
	{x_0} \&\& X \&\& {x_1} \\
	{y_0} \&\& Y \&\& {y_1}
	\arrow[from=1-1, to=2-1]
	\arrow["\sim"{description}, two heads, from=1-3, to=1-1]
	\arrow["\sim"{description}, two heads, from=1-3, to=1-5]
	\arrow[from=1-3, to=2-3]
	\arrow[from=1-5, to=2-5]
	\arrow["\sim"{description}, two heads, from=2-3, to=2-1]
	\arrow["\sim"{description}, two heads, from=2-3, to=2-5]
\end{tikzcd}\]
The same information may also be presented by a single commutative square
\[\begin{tikzcd}[ampersand replacement=\&]
	X \&\& Y \\
	\\
	{x_0 \times x_1} \&\& {x_1 \times y_1}
	\arrow[from=1-1, to=1-3]
	\arrow[two heads, from=1-1, to=3-1]
	\arrow[two heads, from=1-3, to=3-3]
	\arrow[from=3-1, to=3-3]
\end{tikzcd}\]
where we identified the objects, i.e the Reedy fibrant spans, with a single map into a product.
\end{remark}

Note that $P\mathcal{T}$ inherits a fibration category structure such that the two projection functor $\partial_0, \partial_1:  P\mathcal{T} \to \mathcal{T}$ are exact functors.

As an important first step, we would like to construct exact functors between fibration categories from homotopical functors that fail to preserve fibrations. This is the content of the following lemma: 

\begin{lemma}
\label{corr}
 Let $F : \mathcal{C} \to \mathcal{D}$ be a functor between fibration categories. Assume that $F$ maps pullbacks along fibrations to pullback squares that are homotopy pullbacks, that $F(*_\mathcal{C}) \to *_\mathcal{D}$ is a weak equivalence and that weak equivalences are sent to weak equivalences by $F$. Then, in the following pullback square, 
\[\begin{tikzcd}[ampersand replacement=\&]
	{\mathcal{C}'} \&\& {P\mathcal{D}} \\
	\\
	{\mathcal{C} \times \mathcal{D}} \&\& {\mathcal{D} \times \mathcal{D}}
	\arrow["u", from=1-1, to=1-3]
	\arrow["{<\alpha,\beta>}"', from=1-1, to=3-1]
	\arrow["\lrcorner"{anchor=center, pos=0.125}, draw=none, from=1-1, to=3-3]
	\arrow["{<\partial_0,\partial_1>}", from=1-3, to=3-3]
	\arrow["{F \times id_{\mathcal{D}}}"', from=3-1, to=3-3]
\end{tikzcd}\]
computed in $\mathbf{Cat}$, $\mathcal{C}'$ inherits a  fibration category structure and the two projections $\alpha$ and $\beta$ are exact functors.

Moreover, if $P\mathcal{D}$ is a path-object for $\mathcal{D}$, with a map $\iota  : \mathcal{D} \to P\mathcal{D}$, then we have the following diagram of fibration categories
\[\begin{tikzcd}[ampersand replacement=\&]
	\& {\mathcal{C}} \\
	{\mathcal{C}} \&\& { \mathcal{D}} \\
	\& {\mathcal{C}'}
	\arrow["{id_\mathcal{C}}"', from=1-2, to=2-1]
	\arrow["F", from=1-2, to=2-3]
	\arrow["m"{description}, dashed, from=1-2, to=3-2]
	\arrow["\alpha", from=3-2, to=2-1]
	\arrow["\beta"', from=3-2, to=2-3]
\end{tikzcd}\]
where $m : \mathcal{C} \to \mathcal{C}'$ is a weak equivalence.
\end{lemma}

\begin{proof}
Define the class of fibrations of $\mathcal{C}'$ to be those morphisms $f :x \to y$ such that the component $\alpha(f) : x_\mathcal{C} \to y_\mathcal{C}$ on $\mathcal{C}$ and the component $\beta(f) : x_\mathcal{D} \to y_\mathcal{D}$ on $\mathcal{D}$ are fibrations, and, moreover, such that the map $u(f)$ in $P\mathcal{D}$ corresponds, modulo the identification discussed in \cref{span_to_product}, to a square

\[\begin{tikzcd}[ampersand replacement=\&]
	X \&\&\& Y \\
	\& P \\
	\\
	{F(x_\mathcal{C}) \times x_\mathcal{D}} \&\&\& {F(y_\mathcal{C}) \times y_\mathcal{D}}
	\arrow[from=1-1, to=1-4]
	\arrow[two heads, from=1-1, to=2-2]
	\arrow[from=1-1, to=4-1]
	\arrow[from=1-4, to=4-4]
	\arrow[dashed, from=2-2, to=1-4]
	\arrow[dashed, from=2-2, to=4-1]
	\arrow["\lrcorner"{anchor=center, pos=0.125}, draw=none, from=2-2, to=4-4]
	\arrow[from=4-1, to=4-4]
\end{tikzcd}\]
where the gap map $X \to P$ is a fibration (in $\mathcal{D}$). Note that the latter condition would reduce to Reedy fibrancy in $P\mathcal{D}$ if $F$ had the property of preserving fibrations.
 Pullbacks along fibrations in $\mathcal{C}'$ are constructed using the corresponding pullbacks in $P\mathcal{D}$, $\mathcal{D}$ and $\mathcal{C}$ as in the diagram below,
 
 \vspace{2em}

 \adjustbox{scale=0.65}{
 \begin{tikzcd}[ampersand replacement=\&]
	\&\&\& X \&\&\&\& Y \\
	{Q'} \&\&\&\& P \\
	\& Q \&\&\&\& Z \\
	\&\&\& {F(x_\mathcal{C}) \times x_\mathcal{D}} \&\&\&\& {F(y_\mathcal{C}) \times y_\mathcal{D}} \\
	\\
	\& {F(x_\mathcal{C} \times_{y_\mathcal{C}} z_\mathcal{C}) \times (x_\mathcal{D} \times_{y_\mathcal{D}} z_\mathcal{D})} \&\&\&\& {F(z_\mathcal{C}) \times z_\mathcal{D}}
	\arrow[from=1-4, to=1-8]
	\arrow[two heads, from=1-4, to=2-5]
	\arrow[from=1-4, to=4-4]
	\arrow[from=1-8, to=4-8]
	\arrow[from=2-1, to=1-4]
	\arrow["\ulcorner"{anchor=center, pos=0.125, rotate=45}, draw=none, from=2-1, to=2-5]
	\arrow[two heads, from=2-1, to=3-2]
	\arrow[dashed, from=2-5, to=1-8]
	\arrow[dashed, from=2-5, to=4-4]
	\arrow["\lrcorner"{anchor=center, pos=0.125}, draw=none, from=2-5, to=4-8]
	\arrow[from=3-2, to=2-5]
	\arrow[from=3-2, to=3-6]
	\arrow[from=3-2, to=6-2]
	\arrow["\ulcorner"{anchor=center, pos=0.125}, draw=none, from=3-2, to=6-6]
	\arrow[from=3-6, to=1-8]
	\arrow[from=3-6, to=6-6]
	\arrow[from=4-4, to=4-8]
	\arrow["\ulcorner"{anchor=center, pos=0.125}, draw=none, from=4-4, to=6-6]
	\arrow[from=6-2, to=4-4]
	\arrow[from=6-2, to=6-6]
	\arrow[from=6-6, to=4-8]
\end{tikzcd}}

\vspace{2em}
where $Q = P \times_Y Z$ since the bottom square is a pullback by assumption, and hence $Q'$ is the pullback $X \times_Y Z$ by pasting the pullback along of $Q$ along the fibration $X \to P$. Note that $Q' \to x_\mathcal{D} \times_{y_\mathcal{D}} z_\mathcal{D}$ is a weak equivalence as the mediating map between equivalent cospans involving a fibration, and $Q' \to F(x_\mathcal{C} \times_{y_\mathcal{C}} z_\mathcal{C})$ is also a weak equivalence because $F(x_\mathcal{C} \times_{y_\mathcal{C}} z_\mathcal{C}) \simeq F(x_\mathcal{C}) \times_{F(y_\mathcal{C})} F(z_\mathcal{C})$ is a homotopy pullback by assumption on $F$.
Thus, this yields a notion of fibration that is stable under pullbacks.

Moreover, all objects of $\mathcal{C}'$ are fibrant. Indeed, the span 

\[\begin{tikzcd}[ampersand replacement=\&]
	\& {F(*_\mathcal{C})} \\
	{F(*_\mathcal{C})} \&\& {*_\mathcal{D}}
	\arrow["\sim"{description}, two heads, from=1-2, to=2-1]
	\arrow["\sim"{description}, two heads, from=1-2, to=2-3]
\end{tikzcd}\]
defines an object in $P\mathcal{D}$, and consequently in $\mathcal{C}'$, which is terminal, and all objects of $\mathcal{C}'$ are fibrant under the previous definition of fibrations.

We define the weak equivalences componentwise. It is clear that, with this definition, the class of weak equivalences enjoy the $2$-out-of-$3$ property. Moreover, since $F$ preserves weak equivalences, a trivial fibration $f : x \to y$ is such that, in particular, the gap $X \to P$ is a trivial fibration in $\mathcal{D}$. From the construction of pullbacks in $\mathcal{C}'$ detailed above, it follows that the trivial fibrations are stable under pullback.

To conclude, we must check that every morphism $f : x \to y$ factors as a weak equivalence followed by a fibration. To do so, we first form the factorization of $\alpha(f)$ and $\beta(f)$ in $\mathcal{D}$ (through some object $z_\mathcal{D}$) and $\mathcal{C}$ (through $z_\mathcal{C}$), then we extend the factorization as follows:
\[\begin{tikzcd}[ampersand replacement=\&]
	X \& Z \& {P_z} \&\& Y \\
	\\
	\\
	{F(x_\mathcal{C}) \times x_\mathcal{D}} \&\& {F(z_\mathcal{C}) \times z_\mathcal{D}} \&\& {F(y_\mathcal{C}) \times y_\mathcal{D}}
	\arrow["\sim"{description}, from=1-1, to=1-2]
	\arrow[from=1-1, to=4-1]
	\arrow[two heads, from=1-2, to=1-3]
	\arrow[from=1-3, to=1-5]
	\arrow[from=1-3, to=4-3]
	\arrow["\lrcorner"{anchor=center, pos=0.125}, draw=none, from=1-3, to=4-5]
	\arrow[from=1-5, to=4-5]
	\arrow[from=4-1, to=4-3]
	\arrow[from=4-3, to=4-5]
\end{tikzcd}\]
where, in order to make the observation that $Z \to F(z_\mathcal{C}) \times z_\mathcal{D}$ indeed gives a span of trivial fibrations, we crucially rely on the fact that $F$ preserves weak equivalences.

By construction, the projections functors $\alpha$ and $\beta$ are exact. Moreover, with the additional assumption on $P\mathcal{D}$, we have the following diagram
\[\begin{tikzcd}[ampersand replacement=\&]
	{\mathcal{C}} \&\& {\mathcal{D}} \\
	\& {\mathcal{C}'} \&\& {P\mathcal{D}} \\
	\\
	\& {\mathcal{C} \times \mathcal{D}} \&\& {\mathcal{D} \times \mathcal{D}}
	\arrow["F", from=1-1, to=1-3]
	\arrow["m"', dashed, from=1-1, to=2-2]
	\arrow["{<id_\mathcal{C},F>}"', curve={height=24pt}, from=1-1, to=4-2]
	\arrow["\iota", from=1-3, to=2-4]
	\arrow["u", from=2-2, to=2-4]
	\arrow["{<\alpha,\beta>}"', from=2-2, to=4-2]
	\arrow["\lrcorner"{anchor=center, pos=0.125}, draw=none, from=2-2, to=4-4]
	\arrow["{<\partial_0,\partial_1>}", from=2-4, to=4-4]
	\arrow["{F \times id_{\mathcal{D}}}"', from=4-2, to=4-4]
\end{tikzcd}\]
where $m$ is a DK-equivalence, since $\alpha$ is one, and fits in the stated diagram.
\end{proof}

\begin{remark}
 The condition of $F$ preserving homotopy pullbacks in the lemma above can be deduced, in practice, from $F$ mapping fibrations to weak fibrations (i.e maps that yield homotopy pullbacks when pulled back along).
 
 This is in particular the case of the Yoneda embedding described below (\cref{fibcat_ye}).
\end{remark}

\begin{remark}
 The existence of a morphism $\iota : \mathcal{S} \to P\mathcal{S}$ making $P\mathcal{S}$ a path-object for $S$ in the category $\mathbf{FibCat}$ can be derived from the existence of a functorial weak-equivalence/fibration factorization in $\mathcal{S}$. This is notably the case if $\mathcal{S}$ is semi-simplicial in the sense of \cite[Definition 3.1]{ks2019internal}. 
\end{remark}

\section{The Yoneda embedding}
\label{fibcat_ye}

In this section, we aim to obtain an exact version of the Yoneda embedding, which should be a typical example of flat functor.

Consider a fibration category $\mathcal{C}$. Following the construction of Cisinski (3.3 in \cite{cisinski2010invariance}), there is a bicategory $\mathbf{B} \mathcal{C}$ with the same objects as $\mathcal{C}$, and whose hom-objects between $x,y \in \text{Ob}\mathcal{C}$, are the categories of "spans" from $x$ to $y$ with the left leg a trivial fibration.

\begin{definition}
\label{LC_fibcat_bicat}
 We define $L_C(\mathcal{C})$ to be the simplicially enriched category obtained from $\mathbf{B} \mathcal{C}$ by applying the $2$-nerve functor $$N : \mathbf{Bicat} \to \mathbf{Cat}_\Delta$$
 introduced in \cite{lack20062} (Section 3), taking the category of bicategories and normal pseudofunctors to the category of simplicially enriched categories (which coincides with the full subcategory of $\mathbf{Cat}^{\Delta^{op}}$ spanned by the simplicial objects $X$ in $\mathbf{Cat}$ such that $X_0$ is discrete).
\end{definition}

In particular, the authors of \cite{lack20062} establish that the $2$-nerve $N(B)$ of a bicategory $B$, thought of as a simplicial object in $\mathbf{Cat}$, satisfies the Segal condition and is Reedy fibrant (see Section 6 and 7 there). As a consequence, the hom-spaces in the $(\infty,1)$-category $N(B)$ are equivalent to the nerve of the hom-categories in $B$, thus, they have the correct homotopy type by Proposition 3.23 of \cite{cisinski2010invariance}. We rephrase this as follows:

\begin{proposition}
 $L_C(\mathcal{C})$ is a simplicial localization of $\mathcal{C}$, in that the canonical morphism
 $$L_C(\mathcal{C}) \to L_H(\mathcal{C})$$
 is a weak equivalence of simplicially enriched categories, where $L_H(\mathcal{C})$  is the usual hammock localization of $\mathcal{C}$.
\end{proposition}

\begin{remark}
 Given a relative category $\mathbf{C}$, both the categories of simplicial presheaves $\mathbf{SSet}^{L(\mathbf{C})^{op}}$ on a simplicial localization of $\mathbf{C}$ and the category of simplicial presheaves $\mathbf{SSet}^{\mathbf{C}^{op}}$ admit model structures presenting the quasicategory of $\infty$-presheaves on $\mathbf{Ho}_\infty(\mathbf{C})$ (this is discussed in \cite{dwyer1987equivalences, dwyer1980simplicial}). For our purpose, it will be more convenient to consider functors $$\mathbf{C}^{op} \to \mathbf{SSet}$$ rather than simplicial functors $$L(\mathbf{C})^{op} \to \mathbf{SSet}$$ since they are easier to construct and keep track of. In exchange for that, the model structure on $\mathbf{SSet}^{\mathbf{C}^{op}}$ to consider is not the injective model structure but a left Bousfield localization of it (mentioned below), which still enjoys good properties.
\end{remark}

We now consider the mapping $H : (c',c) \mapsto Hom_{\mathbf{B}\mathcal{C}}(c',c)$ from $\mathcal{C}^{op} \times \mathcal{C}$ to $\mathbf{Cat}$. It is functorial in the second argument and preserves limits, but is only pseudofunctorial in the first. Therefore, we will rely on a strictification procedure to replace it with an equivalent strict functor.

\begin{definition}
\label{func_homf}
 We define $H' : \mathcal{C}^{op} \times \mathcal{C}$ to $\mathbf{Cat}$ by mapping $(c',c)$ to the category whose objects are the pairs $(f : c' \to d', s \in Hom_{\mathbf{B}\mathcal{C}}(d',c))$ and whose arrows from $(f,s)$ to $(g,t)$ are the arrows $f^* s \to g^* t$ in the category $Hom_{\mathbf{B}\mathcal{C}}(c',c))$.
\end{definition}

\begin{lemma}
 The mapping $H'$ defines a (strict) functor $\mathcal{C}^{op} \times \mathcal{C} \to \mathbf{Cat}$ that is equivalent to $H$ and preserves limits in the second argument.
\end{lemma}

\begin{proof}
 The definition of $H'$ is essentially what we get by applying the usual strictification procedure for pseudofunctors into $\mathbf{Cat}$, from \cite{power1989general}. That $H'$ preserves limits in the second argument is obvious by definition because $H$ enjoys this property.
\end{proof}

Applying the nerve functor $\mathbf{N} : \mathbf{Cat} \to \mathbf{SSet}$ and transposing, this yields a functor
$$\mathbf{y}_0 : \mathcal{C} \to \mathbf{SSet}^{\mathcal{C}^{op}}$$
that provides a candidate for a Yoneda embedding, although it does not define an exact functor, mainly because it does not map fibrations in $\mathcal{C}$ to injective fibrations in $\mathbf{SSet}^{\mathcal{C}^{op}}$. However, it does map them to weak fibrations like the usual Yoneda embedding $\mathcal{C} \to \mathbf{SSet}^{\mathcal{C}^{op}}$: this is established in \cite[Corollaire 3.12]{cisinski2010invariance} (see also Proposition 3.6 there).

To work around this, our goal is now to apply \cref{corr} to get a Yoneda embedding in the form of a span of exact functors. While we can compose $\mathbf{y}_0$ with a pointwise fibrant replacement, we still need to address the mismatch between pointwise (i.e, projective) fibrations and injective fibrations of simplicial presheaves. To do so, we make use of the following result due to Shulman:

\begin{proposition}
\label{injective_frep}
 Given a simplicial category $\mathcal{C}$, there exists a limit preserving endofunctor 
 $$\mathbf{R}_{\mathcal{C}} : \mathbf{SSet}^{\mathcal{C}^{op}} \to \mathbf{SSet}^{\mathcal{C}^{op}}$$
 mapping objects $X \in \mathbf{SSet}^{\mathcal{C}^{op}}$ to fibrant objects $\mathbf{R}_{\mathcal{C}}X$, and mapping pointwise fibrations to injective fibrations. 
\end{proposition}

\begin{proof}
 This follows from Corollary 8.16 of \cite{shulman2019} (and the observation right after this result), instantiated with the injective model structures on the categories of diagrams $\mathcal{M} := \mathbf{SSet}^{\mathcal{C}^{op}}$ and $\mathcal{N} := \mathbf{SSet}^{\text{Ob}\mathcal{C}^{op}}$, which are indeed simplicial model structures, and both the right Kan extension functor $G$ and the restriction functor $U$ preserve limits. The condition that the counit $UG \to Id$ is a Quillen fibration (\cite[Definition 8.6, Remark 8.9]{shulman2019}) follows from \cite[Examples 8.4, 8.13 and Remark 8.9]{shulman2019}. We also consider the functor $R : \mathcal{M} \to \mathcal{M}$  derived from the usual fibrant replacement $\mathbf{Ex}_\infty : \mathbf{SSet} \to \mathbf{SSet}$ in a pointwise fashion. Since the latter preserves limits and maps Kan fibrations to Kan fibrations, and since the cobar construction $C(G,UG,U-)$ (with the notations of \cite{shulman2019}), followed by its totalization, defines a limit-preserving endofunctor of $\mathcal{M}$, the composite $\mathbf{R}_{\mathcal{C}} : X \mapsto C(G,UG,URX)$ takes arbitrary objects to fibrant objects and pointwise fibrations to injective fibrations. 
\end{proof}

As a consequence, the composite
$$\mathbf{y} := \mathbf{R}_{\mathcal{C}} \circ \mathbf{y}_0 : \mathcal{C} \to \mathbf{SSet}^{\mathcal{C}^{op}}$$
takes its values in the subcategory of fibrant simplicial presheaves and preserves pullbacks (although it does not preserve the terminal object nor the fibrations), so \cref{p_fibc} may now be introduced, following the result of \cref{corr}.

We first define $\mathcal{P}(\mathcal{C})$ to be the full subcategory of fibrant objects for the left Bousfield localization of the injective model structure for simplicial presheaves on $\mathcal{C}$ with respect to the weak equivalences of $\mathcal{C}$ (or, rather, the image of these weak equivalences through the Yoneda embedding $h : \mathcal{C} \to \mathbf{SSet}^{\mathcal{C}^{op}}$), as considered in Section 3 of \cite{cisinski2010invariance} under the name $P_w'(\mathcal{C})$.

\begin{remark}
 The above construction makes sense for any relative category $\mathbf{C}$. We will also denote it $\mathcal{P}(\mathbf{C})$. However, the construction below only makes sense when we start from a fibration category.
\end{remark}

\begin{definition}
\label{p_fibc}
 We define  $\mathcal{Q}(\mathcal{C})$ as the following pullback in $\mathbf{Cat}$:
 
\[\begin{tikzcd}[ampersand replacement=\&]
	{\mathcal{Q}(\mathcal{C})} \&\& {P\mathcal{P}(\mathcal{C})} \\
	\\
	{\mathcal{C} \times \mathcal{P}(\mathcal{C})} \&\& {\mathcal{P}(\mathcal{C}) \times \mathcal{P}(\mathcal{C})}
	\arrow[from=1-1, to=1-3]
	\arrow[from=1-1, to=3-1]
	\arrow["\lrcorner"{anchor=center, pos=0.125}, draw=none, from=1-1, to=3-3]
	\arrow["{<\pi_0,\pi_1>}", from=1-3, to=3-3]
	\arrow["{\mathbf{y} \times id_{\mathcal{P}(\mathcal{C})}}"', from=3-1, to=3-3]
\end{tikzcd}\]
\end{definition}

\begin{lemma}[The Yoneda embedding for fibration categories]
\label{fibcat_yoneda}
 $\mathcal{Q}(\mathcal{C})$ is a fibration category, and the projections $\mathcal{Q}(\mathcal{C}) \to \mathcal{C} \times \mathcal{P}(\mathcal{C})$ are exact functors.

\end{lemma}

\begin{proof}
 This is an instance of \cref{corr} since $\mathbf{y} : \mathcal{C} \to \mathcal{P}(\mathcal{C})$ satisfies the required assumptions.
\end{proof}

\begin{remark}
 Since $\mathbf{y} : \mathcal{C} \to \mathcal{P}(\mathcal{C})$ factors through the sub-fibration category $\mathbf{R}(\mathcal{C})$ of essentially representable simplicial presheaves, we also have a restricted version of the Yoneda embedding as follows:

\[\begin{tikzcd}[ampersand replacement=\&]
	{\mathbf{Q}(\mathcal{C})} \&\& {P\mathbf{R}(\mathcal{C})} \\
	\\
	{\mathcal{C} \times \mathbf{R}(\mathcal{C})} \&\& {\mathbf{R}(\mathcal{C}) \times \mathbf{R}(\mathcal{C})}
	\arrow[from=1-1, to=1-3]
	\arrow[from=1-1, to=3-1]
	\arrow["\lrcorner"{anchor=center, pos=0.125}, draw=none, from=1-1, to=3-3]
	\arrow["{<\pi_0,\pi_1>}", from=1-3, to=3-3]
	\arrow["{\mathbf{y} \times id_{\mathbf{R}(\mathcal{C})}}"', from=3-1, to=3-3]
\end{tikzcd}\]
\end{remark}

Note that this Yoneda embedding indeed presents the correct $\infty$-functor because the (nerves of the) categories $Hom_{\mathbf{B}_0\mathcal{C}}(c',c)$ have the homotopy type of the space of morphisms between $c'$ and $c$ (Proposition 3.23 of \cite{cisinski2010invariance}), and because equivalence of categories yields weak equivalences of simplicial sets between the nerves.

\section{Some technical tools}

In this section, we establish some technical results required to apply our rigidification procedure to flat $\infty$-functors in the subsequent section.

\subsection{From filtered quasicategories to filtered categories}

We first need a functorial version of the construction found in \cite[Theorem 9.1.6.2]{kerodon}, which serves as the cornerstone of our rigidification approach.

\begin{definition}
 We write $\mathbf{Filt}_\infty$ for the full subcategory  of $\mathbf{SSet}$ spanned by the filtered quasicategories and $\mathbf{FiltPos}$ for the full subcategory of $\mathbf{Cat}$ spanned by the filtered posets.

 We write $\iota_\mathbf{Filt} : \mathbf{FiltPos} \to \mathbf{Filt}_\infty$ for the canonical inclusion.
\end{definition}

\begin{proposition}
\label{func_filt}
 There exists a functor $$\rho_\mathbf{Filt} : \mathbf{Filt}_\infty \to \mathbf{FiltPos}$$
 and natural transformation $$\Phi : \iota_\mathbf{Filt} \circ \rho_\mathbf{Filt} \to id_{\mathbf{Filt}_\infty}$$
 whose components are final $\infty$-functors.
 
 Moreover, $\rho_\mathbf{Filt}$ fits in a diagram

\[\begin{tikzcd}[ampersand replacement=\&]
	{\mathbf{Filt}_\infty} \&\& {\mathbf{FiltPos}} \\
	\\
	{\mathbf{SSet}} \&\& {\mathbf{Cat}}
	\arrow["{\rho_\mathbf{Filt}}", from=1-1, to=1-3]
	\arrow[hook', from=1-1, to=3-1]
	\arrow[hook', from=1-3, to=3-3]
	\arrow["\rho"', from=3-1, to=3-3]
\end{tikzcd}\]
 where $\rho$ preserves colimits.
\end{proposition}

\begin{proof}
The action on objects is the one provided by the construction of Theorem 9.1.6.2 in \cite{kerodon}, modulo a slight variation that does not change the argument, and enables functoriality of the construction. Namely, we map a filtered quasicategory $\mathcal{A}$ to the poset of finite simplicial subsets of $\mathcal{A} \times \mathbf{N}(\omega)$ that are of the form $K^\triangleright$, and where we additionally require the label of the apex (the natural number $n \in \mathbf{N}(\omega)$) to be strictly greater than the labels of the remaining vertices.

 We will check that, relying on this additional constraint, the construction can be promoted to a functor as stated.

 We first consider the  functor $\mathbf{Sub} : \mathbf{SSet}^{op} \to \mathbf{Pos}$ mapping a simplicial set $S$ to the poset $\mathbf{Sub}(S)$ of its simplicial subsets. 
 Given a map $k : S \to T$ in  $\mathbf{SSet}$, the precomposition functor $$k^* := \mathbf{Sub}(k) : \mathbf{Sub}(T) \to \mathbf{Sub}(S)$$
 admits a left adjoint: $$\exists_k : \mathbf{Sub}(S) \to \mathbf{Sub}(T)$$

 Thanks to the uniqueness of the left adjoints (between posets), we also get a covariant functor:
 $$\mathbf{Sub}_\exists : \mathbf{SSet} \to \mathbf{Pos}$$

 We can now define the functors $\rho$ (and $\rho_\mathbf{Filt}$) by mapping a (filtered) quasicategory $\mathcal{A}$ to the sub-poset of $\mathbf{Sub}_\exists(\mathcal{A} \times \mathbf{N}(\omega))$ spanned by the simplicial subsets of the form described above. Following the proof of Theorem 9.1.6.2 in \cite{kerodon}, the latter is indeed a filtered poset when $\mathcal{A}$ is filtered. We still need to check that, given an $\infty$-functor $k : \mathcal{A} \to \mathcal{B}$ the direct image action $$\exists_k : \mathbf{Sub}(\mathcal{A} \times \mathbf{N}(\omega)) \to \mathbf{Sub}(\mathcal{B} \times \mathbf{N}(\omega))$$ is compatible with the previous construction (that is, it maps $\rho_\mathbf{Filt}(\mathcal{A})$ to $\rho_\mathbf{Filt}(\mathcal{B})$).
 By definition, a simplicial subset $K^\triangleright \subset \mathcal{A} \times \mathbf{N}(\omega)$ in $\rho_\mathbf{Filt}(\mathcal{A})$, is mapped by $\exists_k$ to the simplicial subset obtained as the following image:
 
\[\begin{tikzcd}[ampersand replacement=\&]
	{K^\triangleright} \&\& {\mathcal{A} \times \mathbf{N}(\omega)} \&\& {\mathcal{B} \times \mathbf{N}(\omega)} \\
	\\
	\&\& {L'}
	\arrow[from=1-1, to=1-3]
	\arrow[two heads, from=1-1, to=3-3]
	\arrow[from=1-3, to=1-5]
	\arrow[hook, from=3-3, to=1-5]
\end{tikzcd}\]
However, since $K \to L'$ is an epimorphism and $K^\triangleright$ is finite, $L'$ must be finite. Moreover, the cone point $x$ of $K^\triangleright$ is mapped to a cone point $y$, incidentally showing that $L' = L^\triangleright$ for some finite $L$. 

Indeed, given a $n$-simplex $\sigma$ of $L'$ whose vertices are different from $y$, there exists an $n$-simplex $\sigma'$ of $K^\triangleright$ (hence actually of $K$) which is mapped by $K^\triangleright \to L'$ to $\sigma$, since this morphism is epi. There is a unique $n+1$-simplex $\theta'$ of $K^\triangleright$ with last vertex $x$ and such that the first $n$-vertices inclusion yields $\sigma' \subset \theta'$. The image of this simplex extends $\sigma$ to a $(n+1)$-simplex $\theta$. It is the unique such simplex with the property that its last vertex is $y$ and that the inclusion of the first $n$ vertices is given by $\sigma$, because $y$ has not been collapsed to another vertex (since its natural number label is still strictly greater than those of the other vertices). By definition of the construction $(-)^\triangleright$, this precisely means that $L' = L^\triangleright$, where $L$ is the full simplicial subset of $L'$ spanned by the vertices different from $y$.
This concludes the description of the functor $\rho_\mathbf{Filt}$.

The final $\infty$-functors $\rho_\mathbf{Filt}(\mathcal{A}) \to \mathcal{A}$ obtained canonically from the cone point projection, as discussed in the proof of Theorem 9.1.6.2 in \cite{kerodon}, can indeed be checked to assemble into a natural transformation,  since the direct image functors preserve the cone point of the simplicial subset of the form $K^\triangleright$ as seen above.

Commutation with coproduct follows from the fact that a diagram of the form $K^\triangleright$ in $\mathcal{A} \amalg \mathcal{B}$ must lie entirely in either component of the coproduct (this is because the diagram is connected). 

It is also clear that $\rho$ preserves filtered colimits since the subobjects considered in the definition are finite. This actually entails that $\rho$ preserves reflexive coequalizers. Indeed, applying \cite[Theorem 1.3]{barr1974right}, it is enough to show that $\rho$ maps coequalizers of kernel pairs to coequalizers. Hence, we consider a regular epimorphism $f : X \to Y$ in $\mathbf{SSet}$ and we will show that the following diagram is a coequalizer: 
\[\begin{tikzcd}[ampersand replacement=\&]
	{\rho (X \times_Y X)} \&\& {\rho X} \&\& {\rho Y}
	\arrow["{\rho\pi_0}", shift left=2, from=1-1, to=1-3]
	\arrow["{\rho\pi_1}"', shift right=2, from=1-1, to=1-3]
	\arrow["{\rho f}"', from=1-3, to=1-5]
\end{tikzcd}\]
It is clear that any two diagrams $K^\triangleright \to X$ and $(K')^\triangleright \to X$ identified by the equivalence relation generated from $\rho \pi_0$ and $\rho \pi_1$ have their simplices identified by $\pi_0$ and $\pi_1$. Conversely, suppose two such diagrams are mapped to the same diagram in $\rho Y$. By definition, we have a map between spans as below, that is pointwise a monomorphism.
\[\begin{tikzcd}[ampersand replacement=\&]
	{(K \times_L K')^\triangleright } \&\&\& {X\times_Y X} \\
	\& {K^\triangleright \times_{L^\triangleright} (K')^\triangleright} \&\&\&\& X \\
	\&\& X \& {(K')^\triangleright} \\
	{K^\triangleright} \&\&\&\& Y \\
	\&\& {L^\triangleright}
	\arrow[dashed, hook, from=1-1, to=2-2]
	\arrow[dashed, from=1-1, to=4-1]
	\arrow[dashed, from=1-4, to=2-6]
	\arrow[dashed, from=1-4, to=3-3]
	\arrow[dashed, hook, from=2-2, to=1-4]
	\arrow[dashed, from=2-2, to=3-4]
	\arrow[dashed, from=2-2, to=4-1]
	\arrow[two heads, from=2-6, to=4-5]
	\arrow[two heads, from=3-3, to=4-5]
	\arrow[hook, from=3-4, to=2-6]
	\arrow[two heads, from=3-4, to=5-3]
	\arrow[hook, from=4-1, to=3-3]
	\arrow[two heads, from=4-1, to=5-3]
	\arrow[hook, from=5-3, to=4-5]
\end{tikzcd}\]
Here, the mediating arrow $(K \times_L K')^\triangleright  \to K^\triangleright \times_{L^\triangleright} (K')^\triangleright$ is a monomorphism, hence so is the diagram $(K \times_L K')^\triangleright \to X \times_Y X$. Moreover, by construction, this diagram is mapped by $\rho \pi_0$ (resp. $\rho \pi_1$) to the subobject $K^\triangleright \to X$ (resp. $(K')^\triangleright \to X$). Therefore, the latter two diagrams are identified by the equivalence relation generated from $\rho \pi_0$ and $\rho \pi_1$. This proves that $\rho Y$ is the quotient of $X$ by this equivalence relation, so that $\rho$ preserves the coequalizer of the kernel pair we started from.

Finally, $\rho$ preserves all (small) colimits because it preserves coproducts and reflexive coequalizers.
\end{proof}

\subsection{Fibration categories in $\mathbf{RelCat}$}

We will consider the Barwick-Kan model structure on the category $\mathbf{RelCat}$ of relative categories (\cite{barwick2012relative}). Fibration categories are fibrant (this is established in \cite{meier2016fibration}), consequently, any map $I \to \mathcal{F}$ in the homotopy category $\mathbf{Ho}(\mathbf{RelCat})$ with a codomain that is a fibration category may be represented by a span of homotopical functors:
\[\begin{tikzcd}[ampersand replacement=\&]
	\& J \\
	I \&\& {\mathcal{F}}
	\arrow[from=1-2, to=2-1]
	\arrow[from=1-2, to=2-3]
\end{tikzcd}\]
We will show that this result can be strengthened as follows.

\begin{proposition}
\label{fibcat_fibrant}
 Let $\mathcal{F}$ be a fibration category and $I$ be an inverse homotopical category with finite coslice $I_{i/}$. Then any map $I \to \mathcal{F}$ in the homotopy category $\mathbf{Ho}(\mathbf{RelCat})$ can be represented by a single homotopical functor $I \to \mathcal{F}$.
\end{proposition}

\begin{proof}
By Lemma 5.4 in \cite{barwick2012relative}, the projection $\pi: \xi I \to I$ is a weak equivalence of relative categories, where $\xi$ is the "relative" double subdivision functor. Moreover, $\xi I \simeq \mathbf{K}_\xi \mathbf{N} I$ is cofibrant since every object is cofibrant in the Rekz model structure on simplicial spaces, and since $$\mathbf{K}_\xi : \mathbf{sSpace}_\mathbf{Rekz} \to \mathbf{RelCat}$$ is a left Quillen functor. Therefore, we may consider a span representing the map $I \to \mathcal{F}$ as below.
\[\begin{tikzcd}[ampersand replacement=\&]
	\& {\xi I} \\
	I \&\& {\mathcal{F}}
	\arrow["\pi"', from=1-2, to=2-1]
	\arrow["F", from=1-2, to=2-3]
\end{tikzcd}\]
where $\xi I$ is also an inverse category.
 We consider a Reedy fibrant replacement of $F'$ of $F$, as constructed in \cite[Theorem 9.2.4]{radulescu2006cofibrations} in a slightly different context (namely with a notion of fibration category having all small homotopy limit). However, the fact that $\xi I$ is also an inverse category with finite coslice ensures that the fibrant replacement can be constructed relying only on finite limits, which exist in $\mathcal{F}$.
 
 Now, we may compute the right Kan extension of $F'$ along $\pi$. Indeed, given an object $i \in I$, the value at $i$ is computed as the limit of the diagram $F \circ p_i$ as below

\[\begin{tikzcd}[ampersand replacement=\&]
	{i \downarrow \pi} \&\& {\xi I} \&\& {\mathcal{F}} \\
	\\
	{*} \&\& I
	\arrow["{p_i}", from=1-1, to=1-3]
	\arrow[from=1-1, to=3-1]
	\arrow["F", from=1-3, to=1-5]
	\arrow["\pi"', from=1-3, to=3-3]
	\arrow[between={0.3}{0.7}, Rightarrow, from=3-1, to=1-3]
	\arrow["i"', from=3-1, to=3-3]
\end{tikzcd}\]
where $p_i$ is an discrete opfibration (the fibers at $j \in \xi I$ being the homset $Hom(i,\pi j)$). In particular, the precomposition functor $p_i^* : \mathcal{F}^{\xi I} \to \mathcal{F}^{i \downarrow \pi}$ preserves the matching category construction, mapping $F'$ to a (finite) Reedy fibrant diagram. The Kan extension is moreover a homotopy Kan extension, so that the functor $F_e : I \to \mathcal{F}$ we obtain is mapped by $\mathbf{Ho}_\infty$ to the right Kan extension of $\mathbf{Ho}_\infty F'$ along $\mathbf{Ho}_\infty \pi$ in the $(\infty,1)$-category $\mathcal{QCat}$ of (small) quasicategories. The latter being an isomorphism, it follows that the $2$-cell involved in the Kan extension is invertible, so that $F_e$ indeed represents the correct morphism in the homotopy category $\mathbf{Ho}(\mathbf{RelCat})$.
\end{proof}

There is a fibred analogue to this result, that we obtain at the cost of taking values in a model category:

\begin{proposition}
\label{fibcat_fibrant_fb}
 Let $\mathcal{F}$ be a fibration category and $I \to B$ be an opfibration between homotopical categories, corresponding to a diagram $D : B \to \mathbf{Cat}$. Consider the diagram $D_\xi : \xi \circ D$, together with the natural transformation $D_\xi \to D$ induced by $\pi : \xi \to id_\mathbf{RelCat}$, and to a map $I_\xi \to I$ over $B$ through the Grothendieck construction.

 Then, for any map $I_\xi \to \mathcal{F}$, the functor $F : I_\xi \to \mathcal{F} \to \mathbf{R}(\mathcal{F})$ can be extended to homotopical functor $F_e : I \to \mathbf{R}(\mathcal{F})$.
\end{proposition}

\begin{proof}
 We proceed in similar way as for \cref{fibcat_fibrant}, that is, we first consider an injectively fibrant replacement $F'$ of $F$. Then, we may compute the right Kan extension of the composite $I_\xi \to \mathbf{R}(\mathcal{F}) \to \mathcal{P}(\mathcal{F})$ along $I_\xi \to I$. This is a homotopy Kan extension along a weak equivalence, so the diagram $I \to \mathcal{P}(\mathcal{F})$ we obtain fits in a homotopy commutative triangle (in $\mathbf{RelCat}$)

\[\begin{tikzcd}[ampersand replacement=\&]
	{I_\xi} \&\& {\mathbf{R}(\mathcal{F})} \&\& {\mathcal{P}(\mathcal{F})} \\
	\\
	I
	\arrow[from=1-1, to=1-3]
	\arrow[from=1-1, to=3-1]
	\arrow[from=1-3, to=1-5]
	\arrow[from=3-1, to=1-5]
\end{tikzcd}\]
so that it factors through $\mathbf{R}(\mathcal{F})$, yielding the expected homotopical functor $F_e : I \to \mathbf{R}(\mathcal{F})$. 
\end{proof}

\subsection{Projective cofibrancy from Reedy cofibrancy}

Given a Quillen adjunction $$F : \mathcal{M} \to \mathcal{N} : G$$ between combinatorial model categories, and given a small category $I$, the induced adjunction $$F^I : \mathcal{M}^I \to \mathcal{N}^I : G^I$$ is a Quillen adjunction for both the injective and the projective model structures (this is the content of \cite[Remark A.2.8.6]{lurie2009}).

Given a Reedy category $R$, the Reedy cofibrancy criterion for $\mathcal{M}^R$ is expressed in terms of colimits and cofibrations in  $\mathcal{M}$. In particular, validity of the criterion is preserved by any left Quillen functor $F$ as above; hence, $\mathcal{M}^R \to \mathcal{N}^R$ is also left Quillen (this does not require $\mathcal{M}$ and $\mathcal{N}$ to be combinatorial). Actually, we don't even need to assume that $F$ is a left adjoint: preservation of colimits together with (trivial) cofibrations is enough to ensure to same properties for $F^R$.

The goal of this section is to leverage this observation to deduce a similar criterion for projective model structures (\cref{lquillen_reedy} below).

We rely on the following construction, that allows us to approximate diagrams in a given category $D$ by diagrams on a Reedy category, as discussed in \cite{dwyer2004homotopy}.

\begin{definition}
 For $D$ a small category, the category $\Delta D$ denotes the category of simplices of $D$. Explicitly, $\Delta^{op} D$ has:
 \begin{itemize}
  \item objects the sequences of $n$ composable arrows in $D$ (i.e, the functors $[n] \to D$).
  \item morphisms the maps $[m] \to [n]$ in $\Delta$ making the following diagram commute

\[\begin{tikzcd}[ampersand replacement=\&]
	{[m]} \&\& {[n]} \\
	\\
	\& D
	\arrow[from=1-1, to=1-3]
	\arrow[from=1-1, to=3-2]
	\arrow[from=1-3, to=3-2]
\end{tikzcd}\]
where we rely on the fact that $\Delta$ can be identified with the full subcategory of $\mathbf{Cat}$ spanned by the finite linear orders $[n]$.
\end{itemize}
We also write $\Delta^{op} D$ for the opposite of the category $\Delta D$.
\end{definition}

An object in $\Delta D$ of the form $[1] \to D$ corresponds to a single arrow $f : x \to y$ in $D$. Similarly, an object of the form $[2] \to D$ corresponds to a composable pair of arrows in $D$. Consider the map $d_0 : [1] \to [2]$ mapping $0$ (resp. $1$) to $1$ (resp. $2$). An arrow in $\Delta D$ whose underlying map in $\Delta$ is $d_0$ correspond to a diagram 
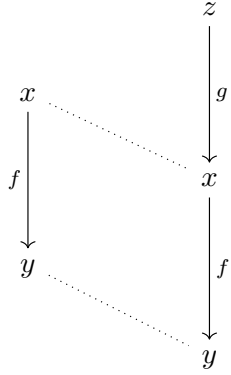
\[\begin{tikzcd}[ampersand replacement=\&]
	\&\& z \\
	x \\
	\&\& x \\
	y \\
	\&\& y
	\arrow["g", from=1-3, to=3-3]
	\arrow[dotted, no head, from=2-1, to=3-3]
	\arrow["f"', from=2-1, to=4-1]
	\arrow["f", from=3-3, to=5-3]
	\arrow[dotted, no head, from=4-1, to=5-3]
\end{tikzcd}\]
where the first object (i.e the single map $f$), coincides via restriction along $d_0$ to the second component of the second object (the pair of composable maps $g$ and $f$).

We observe that, if we project every diagram $F : [n] \to D$ (i.e object of $\Delta D$) to its initial object, the arrows $\alpha : F \to G$ in $\Delta D$ give rise to a unique mediating map $G_0 \to F_0$ corresponding to the unique map in the diagram $G : [n] \to D$ from the first vertex $G_0$ of $G$ to image of the first vertex $F_0$ of $F$ through the simplicial operator corresponding to $\alpha$. This mapping thus extends to a contravariant functor:

\begin{definition}
We write  $p_i : \Delta^{op} D \to D$ for the projection functor mapping each path to its first vertex. For $d$ an object of $D$, define $p_i^{-1}d$ to be the subcategory of the comma category $d \downarrow p_i$ whose objects are the paths $[n] \to D$ whose first vertex is $d$, and whose arrows are the simplicial operators $[m] \to [n]$ that map the first element of $[m]$ to the first element of $[n]$.  

Dually, we write $p_t : \Delta D \to D$ for the last-vertex functor, and we define $p_t^{-1}d$ in a similar way as $p_i^{-1}d$.
\end{definition}

The following results are established in 22-23 of \cite{dwyer2004homotopy}:

\begin{proposition}
\label{approx_reedy}
 With the previous notations:
 
 \begin{itemize}
  \item $p_i^{-1}d$ has a terminal object and is an initial subcategory of $d \downarrow p_i$.
  \item The subcategory of $p_i^{-1}d$ spanned by the increasing Reedy maps (i.e, the surjective simplicial operators), is a disjoint union of categories with initial objects. In particular, any constant diagram of shape $p_i^{-1}d$ is Reedy cofibrant (the latching categories all have an initial object).
  \item The comma category $d \downarrow p_i$ is isomorphic to $\Delta^{op}(D_{d/})$ and the projection functor $\pi_d : d \downarrow p_i \to \Delta^{op} D$, which results from the projection of the coslice  $D_{d/} \to D$, induces a right Quillen precomposition functor 
  $$\pi_d^* :  \mathcal{M}^{\Delta^{op} D} \to \mathcal{M}^{d \downarrow p_i}$$
  for any model category $\mathcal{M}$.
 \end{itemize}
 
 Dually:
 \begin{itemize}
  \item $p_t^{-1}d$ has an initial object and is a final subcategory of $p_t \downarrow d$.
  \item The subcategory of $p_t^{-1}d$ spanned by the increasing Reedy maps (i.e, the surjective simplicial operators), is a disjoint union of categories with terminal objects. In particular, any constant diagram of shape $p_i^{-1}d$ is Reedy fibrant.
  \item The comma category $p_t \downarrow d$ is isomorphic to $\Delta^{op}(D_{/d})$ and the projection functor $\pi_d : p_t \downarrow d \to \Delta^{op} D$, which results from the projection of the slice  $D_{/d} \to D$, induces a left Quillen precomposition functor 
  $$\pi_d^* :  \mathcal{M}^{\Delta^{op} D} \to \mathcal{M}^{p_t \downarrow d}$$
  for any model category $\mathcal{M}$.
 \end{itemize}
\end{proposition}

It will be convenient to use the following terminology: given a lax commutative square of small categories
\[\begin{tikzcd}[ampersand replacement=\&]
	A \&\& B \\
	\\
	C \&\& D
	\arrow["f", from=1-1, to=1-3]
	\arrow["g"', from=1-1, to=3-1]
	\arrow[between={0.3}{0.7}, Rightarrow, from=1-3, to=3-1]
	\arrow["{g'}", from=1-3, to=3-3]
	\arrow["{f'}"', from=3-1, to=3-3]
\end{tikzcd}\]
and given a cocomplete category $\mathcal{M}$, there is a canonical \textit{mate} transformation $\mu$ (induced from the unit/counit of the adjunctions as well as the $2$-cell above) as below:
\[\begin{tikzcd}[ampersand replacement=\&]
	{\mathcal{M}^A} \&\& {\mathcal{M}^B} \\
	\\
	{\mathcal{M}C} \&\& {\mathcal{M}D}
	\arrow["{g_!}"', from=1-1, to=3-1]
	\arrow["{f^*}"', from=1-3, to=1-1]
	\arrow["\mu"', between={0.3}{0.7}, Rightarrow, from=1-3, to=3-1]
	\arrow["{(g')_!}", from=1-3, to=3-3]
	\arrow["{f'^*}", from=3-3, to=3-1]
\end{tikzcd}\]

We say that the former square is (left) ($\mathcal{M}$-)exact if the mate transformation is invertible. Similarly, there is a notion of (right) ($\mathcal{M}$-)exactness for an oplax commutative square of small categories and a complete category $\mathcal{M}$.

The following criterion for testing exactness of a square is useful in practice.

\begin{lemma}
 A lax commutative square of small categories

\[\begin{tikzcd}[ampersand replacement=\&]
	A \&\& B \\
	\\
	C \&\& D
	\arrow[from=1-1, to=1-3]
	\arrow["g"', from=1-1, to=3-1]
	\arrow[between={0.3}{0.7}, Rightarrow, from=1-3, to=3-1]
	\arrow[from=1-3, to=3-3]
	\arrow[from=3-1, to=3-3]
\end{tikzcd}\]
 is exact if and only if, for every object $c \in C$, the composite of the two lax commutative squares is exact:

\[\begin{tikzcd}[ampersand replacement=\&]
	{g \downarrow c} \&\& A \&\& B \\
	\\
	{*} \&\& C \&\& D
	\arrow[from=1-1, to=1-3]
	\arrow[from=1-1, to=3-1]
	\arrow[from=1-3, to=1-5]
	\arrow[between={0.3}{0.7}, Rightarrow, from=1-3, to=3-1]
	\arrow["g"', from=1-3, to=3-3]
	\arrow[between={0.3}{0.7}, Rightarrow, from=1-5, to=3-3]
	\arrow[from=1-5, to=3-5]
	\arrow["c"', from=3-1, to=3-3]
	\arrow[from=3-3, to=3-5]
\end{tikzcd}\] 
\end{lemma}

\begin{corollary}
\label{precomp_reedy_quillen}
For any model category $\mathcal{M}$, the precomposition functor $$p_i^* : \mathcal{M}^{D} \to \mathcal{M}^{\Delta^{op} D}$$
 is a left Quillen functor from the projective model structure to the Reedy model structure.
 Dually, the precomposition functor $$p_t^* : \mathcal{M}^{D} \to \mathcal{M}^{\Delta D}$$
 is a right Quillen functor from the injective model structure to the Reedy model structure.
 Moreover, the composite $$(p_i \circ p_t^{op})^* \mathcal{M}^{D} \to \mathcal{M}^{\Delta^{op} (\Delta D)}$$ is also a left Quillen functor, and the composite $$(p_t \circ p_i)^* \mathcal{M}^{D} \to \mathcal{M}^{\Delta^{op} (\Delta D)}$$ is a right Quillen functor (with the same model structures as above).
\end{corollary}

\begin{proof}
 Let us prove the first statement. For this, it is enough to check that the right Kan extension functor $(p_i)_*$ is right Quillen. To see this, observe that, for any object $d$ of $D$, the following rectangle is exact,
\[\begin{tikzcd}[ampersand replacement=\&]
	{p_i^{-1}d} \&\& {d \downarrow p_i} \&\& {\Delta^{op} D} \\
	\\
	{*} \&\& {*} \&\& D
	\arrow[from=1-1, to=1-3]
	\arrow["{i_d}", curve={height=-12pt}, from=1-1, to=1-5]
	\arrow["{!_d}"', from=1-1, to=3-1]
	\arrow[from=1-3, to=1-5]
	\arrow[from=1-3, to=3-3]
	\arrow["{p_i}", from=1-5, to=3-5]
	\arrow[from=3-1, to=3-3]
	\arrow[between={0.3}{0.7}, Rightarrow, from=3-3, to=1-5]
	\arrow["d"', from=3-3, to=3-5]
\end{tikzcd}\]
 i.e, the associated mate transformation is invertible ($d^* \circ (p_i)_* \simeq  !_* \circ i_d^*$), and the functor $i_d^*$ is right Quillen by \cref{approx_reedy}. Hence, it is enough to observe that $!_*$  is also right Quillen. This is the case since $p_i^{-1}d$ has cofibrant constant as observed in \cref{approx_reedy}, applying Lemma 9.4 and Corollary 9.6 of \cite{riehlreedy}.
 
 For the first part of the second statement, we can argue similarly by considering the following diagram where the bottom left square and the rectangle on the right are exact:

\[\begin{tikzcd}[ampersand replacement=\&]
	\&\& {(p_i^{op} \circ p_t) \downarrow d \simeq \Delta(p_i^{op} \downarrow d)} \&\& { \Delta(\Delta D)} \\
	{(p_i^{op})^{-1}d} \&\& {p_i^{op} \downarrow d} \&\& {\Delta D} \\
	{*} \&\& {*} \&\& {D^{op}}
	\arrow[from=1-3, to=1-5]
	\arrow["{p_t}"', from=1-3, to=2-3]
	\arrow["\ulcorner"{anchor=center, pos=0.125, rotate=45}, draw=none, from=1-3, to=2-5]
	\arrow["{p_t}", from=1-5, to=2-5]
	\arrow[from=2-1, to=2-3]
	\arrow["{!}"', from=2-1, to=3-1]
	\arrow[from=2-3, to=2-5]
	\arrow[from=2-3, to=3-3]
	\arrow[between={0.3}{0.7}, Rightarrow, from=2-5, to=3-3]
	\arrow["{p_i^{op}}", from=2-5, to=3-5]
	\arrow[from=3-1, to=3-3]
	\arrow["d"', from=3-3, to=3-5]
\end{tikzcd}\]

By \cite[22.10 (iii)]{dwyer2004homotopy}, precomposition with the top arrow defines a left Quillen functor, and, by the first part of the proof, the left  Kan extension functor along $p_t : \Delta(p_i^{op} \downarrow d) \to p_i^{op} \downarrow d$ is left Quillen. Restriction along the inclusion $(p_i^{op})^{-1}d \to p_i^{op} \downarrow d$ is also a left Quillen functor, as well as the last colimit functor (i.e left Kan extension along $! : (p_i^{op})^{-1}d \to *$). This proves that left Kan extension along the composite $p_i^{op} \circ p_t :  \Delta(\Delta D) \to D^{op}$ is a left Quillen functor. Applying this with $\mathcal{M}^{op}$ in place of $\mathcal{M}$, we conclude that the right Kan extension functor along $p_i \circ p_t^{op}$ is right Quillen. Thus, the precomposition functor $(p_i \circ p_t^{op})^*$ is left Quillen.

The second part is similar by considering the following diagram:
\[\begin{tikzcd}[ampersand replacement=\&]
	\&\& {(p_t \circ p_i) \downarrow d \simeq \Delta^{op}(p_t \downarrow d)} \&\& { \Delta^{op}(\Delta D)} \\
	{(p_t)^{-1}d} \&\& {p_t \downarrow d} \&\& {\Delta D} \\
	{*} \&\& {*} \&\& D
	\arrow[from=1-3, to=1-5]
	\arrow["{p_i}"', from=1-3, to=2-3]
	\arrow["\ulcorner"{anchor=center, pos=0.125, rotate=45}, draw=none, from=1-3, to=2-5]
	\arrow["{p_i}", from=1-5, to=2-5]
	\arrow[from=2-1, to=2-3]
	\arrow["{!}"', from=2-1, to=3-1]
	\arrow[from=2-3, to=2-5]
	\arrow[from=2-3, to=3-3]
	\arrow[between={0.3}{0.7}, Rightarrow, from=2-5, to=3-3]
	\arrow["{p_t}", from=2-5, to=3-5]
	\arrow[from=3-1, to=3-3]
	\arrow["d"', from=3-3, to=3-5]
\end{tikzcd}\]
\end{proof}

\begin{lemma}
\label{exact_reedy_quillen}
 The following diagram is an exact square (i.e its mate is invertible):

\[\begin{tikzcd}[ampersand replacement=\&]
	{ \Delta^{op} (\Delta D)} \&\& {\Delta D} \&\& D \\
	\\
	{\Delta^{op}D} \&\& D \&\& D \\
	\\
	D \&\& D \&\& D
	\arrow["{p_i}", from=1-1, to=1-3]
	\arrow["{p_t^{op}}"', from=1-1, to=3-1]
	\arrow["{p_t}", from=1-3, to=1-5]
	\arrow["{p_t}", from=1-3, to=3-3]
	\arrow[from=1-5, to=3-5]
	\arrow[between={0.3}{0.7}, Rightarrow, from=3-1, to=1-3]
	\arrow["{p_i}"', from=3-1, to=3-3]
	\arrow["{p_i}"', from=3-1, to=5-1]
	\arrow[from=3-3, to=3-5]
	\arrow[from=3-3, to=5-3]
	\arrow[from=3-5, to=5-5]
	\arrow[from=5-1, to=5-3]
	\arrow[from=5-3, to=5-5]
\end{tikzcd}\]
\end{lemma}

\begin{proof}
 We first prove that the top rectangle is exact. For convenience, we prove that the dual rectangle (i.e the one obtained by applying $\mathbf{op} : \mathbf{Cat} \to \mathbf{Cat}$) is exact. For this, consider an object $d$ of $\Delta D$. By the usual calculus of mates, it will be enough to establish that the functor $$m : p_t \downarrow d \to D^{op}_{/p_i^{op} d} \simeq (D_{p_i^{op} d/})^{op}$$ is final. But, by \cref{approx_reedy}, we know that $p_t \downarrow d$, which is isomorphic to $\Delta((\Delta D)_{d/})$, has a final subcategory $p_t^{-1} d$ with a terminal object.  Moreover, the induced functor $m' : p_t^{-1} d \to (D_{p_i^{op} d/})^{op}$ maps every object to the identity $p_i^{op} d \to p_i^{op} d$, that is, to the terminal object of $(D_{p_i^{op} d/}))^{op}$. Therefore, for any diagram indexed by $(D_{p_i^{op} d/})^{op}$ in some category $\mathbf{C}$, the diagram obtained by precomposition with $m'$ is a constant diagram of contractible shape with value the object of $\mathbf{C}$ associated with the terminal object of $(D_{p_i^{op} d/})^{op}$. Hence its colimit is equal to this object, which is also the colimit of the original diagram indexed by $(D_{p_i^{op} d/})^{op}$. This proves that $m'$ is final, and hence that the top rectangle is exact.

 We still have to prove that the bottom left square is exact. To see this, it is enough to prove that, for every $d \in D$, the comparison functor $d \downarrow p_i \simeq \Delta^{op}(D_{d/})  \to D_{d/}$ is initial. As before, it is enough to show that $p_i^{-1}d \to D_{d/}$ is initial, which can be proven just like for the functor $m'$ above. 
\end{proof}

\begin{proposition}
\label{lquillen_reedy}
 Let $D$ be small category, and let $\mathcal{M}$ and $\mathcal{M}'$ be model categories such that the projective model structures on the categories of $D$-diagrams exist. Consider a  left Quillen functor $F : \mathcal{M} \to \mathcal{M'}$, namely $F$ maps (trivial) cofibrations to (trivial) cofibrations, and suppose that $F$ preserves colimits. Then $F^D$ is also a left Quillen functor.
\end{proposition}

\begin{proof}
Consider a cofibration $c : X \to Y$ in $\mathcal{M}^D$, and write $Fc$ for its image under $F$. By \cref{precomp_reedy_quillen}, the map $(p_i \circ p_t^{op})^* c$ is a Reedy cofibration in $\mathcal{M}^{\Delta^{op}(\Delta D)}$. Since Reedy cofibrancy can be expressed in terms of relative latching maps being cofibrations, and since $F$ preserves this property by assumption, we obtain that $(p_i \circ p_t^{op})^* Fc$ is also a Reedy cofibration in $(\mathcal{M}')^{\Delta^{op}(\Delta D)}$.  By \cref{exact_reedy_quillen}, taking the left Kan extension along $p_t \circ p_i$ yields back the map $Fc$ up to isomorphism, and we know that the result is a cofibration because this left Kan extension functor is left Quillen by \cref{precomp_reedy_quillen}.
\end{proof}

\begin{remark}
 In the previous proposition, we may replace the assumption "$F$ preserves colimits" by the weaker requirement that $F$ commutes with colimits "up to trivial cofibration" in that, for any diagram $P : I \to \mathcal{M}$, the canonical map $\varinjlim  F \circ P \to F(\varinjlim P)$ is a trivial cofibration (although we will not rely on this observation in the rest of this document).
\end{remark}

\section{Rigidifying flat $\infty$-functors}
\label{section3}

It is well known that any presheaf $F :  C^{op} \to \mathbf{Set}$ is canonically the colimit of its "elements". Explicitly, writing $\mathbf{el}(F)$ for the category of elements of $F$, $F$ can be recovered as the colimit of the diagram $\mathbf{el}{F}^{op} \to C \to \mathbf{Set}{C^{op}}$ of representable functors. If $C^{op}$ has finite limits and $F$ is a left exact presheaf, the category of elements is cofiltered, so that the previous colimit is a filtered colimit. Our goal in this section is to apply this characterization to $\infty$-categorical settings, then to rigidify $\infty$-flatness into strict flatness in order to compute a limit-preserving homotopical functor $F_r:  \mathcal{C} \to \mathbf{SSet}$ from any left exact $\infty$-presheaf $F : \mathbf{Ho}_\infty{C} \to \mathcal{S}$, where $\mathcal{C}$ is a fixed fibration category.

 We use the characterization of flatness, as introduced in Definition 5.3.1.7 of \cite{lurie2009}, in terms of cofiltered $(\infty,1)$-category of elements, which we will approximate functorially by a cofiltered $1$-category thanks to \cref{func_filt}.

\subsection{The unparameterized case}

Write $\mathbf{SSet}^\circ$ for the fibration category of fibrant objects in the Kan-Quillen model structure on simplicial sets.

\begin{proposition}
\label{unp_rigidification}
Consider a flat $\infty$-functor $F : \mathbf{Ho}_\infty(\mathcal{C}) \simeq \mathbf{N}_\Delta L_C(\mathcal{C}) \to \mathcal{S}$, where $\mathcal{C}$ is a fibration category.
Then there exists a homotopical functor $F_r : \mathcal{C} \to \mathbf{SSet}^\circ$,  such that $\mathbf{Ho}_\infty(F_r)$ is homotopic to $F$ as a $\infty$-functor. Moreover $F_r$ preserves finite limits and finite homotopy limits.
\end{proposition}

\begin{proof}
 Consider the left fibration of simplicial sets $$\pi_F :  \mathbf{Ho}_\infty(\mathcal{C})_{/F} \to  \mathbf{Ho}_\infty(\mathcal{C})$$
 associated to $F$.

Now, since $F$ is flat (or equivalently, left exact, since $\mathcal{C}$ is a fibration category so that $\mathbf{Ho}_\infty(\mathcal{C})$ is finitely complete), we know by Remark 5.3.2.11 in \cite{lurie2009} that the quasicategory $\mathbf{Ho}_\infty(\mathcal{C})_{/F}$ is cofiltered.

In this special case, we can consider a filtered poset $P_F$ equipped with a final functor $\mathbf{N}(P_F) \to \mathbf{Ho}_\infty(\mathcal{C})_{/F}^{op}$ , by applying \cref{func_filt} (or directly from Theorem 9.1.6.2 in \cite{kerodon}). We now have a map
$$\mathbf{N}(P_F) \to  \mathbf{Ho}_\infty(\mathcal{C})_{/F}^{op} \to  \mathbf{Ho}_\infty(\mathcal{C})^{op}$$

At this point, composing with the Yoneda embedding, we have a map:
$$\mathbf{N}(P_F) \simeq \mathbf{Ho}_\infty(P_F) \to \mathbf{N}_\Delta L_C(\mathcal{C})^{op} \simeq \mathbf{Ho}_\infty(\mathcal{C}^{op}) \to  \mathbf{Ho}_\infty(\mathcal{P}(\mathcal{C}^{op}))$$

Now, because the functor $\mathbf{Ho}_\infty$ defines a DK-equivalence between the model category $\mathbf{RelCat}$ and the Joyal model structure on simplicial sets (see \cite{barwick2011thomason}), there is a map $D_{0,F} : P_F^{op} \to \mathcal{C}$ in the homotopy category $\mathbf{Ho}(\mathbf{RelCat})$ such that $\mathbf{Ho}_\infty(D_{0,F})$ coincides with the map $\mathbf{N}(P_F^{op}) \to \mathbf{Ho}_\infty(\mathcal{C})$ considered above. Applying \cref{fibcat_fibrant} (observing that $P_F^{op}$ satisfies the corresponding assumptions), this map can be realized by a single morphism, also denoted $D_{0,F} : P_F^{op} \to \mathcal{C}$, in the category $\mathbf{RelCat}$.

Therefore, we may consider the diagram:
$$D_F : P_F \to \mathcal{C}^{op} \to \mathcal{P}(\mathcal{C}^{op}) \subset \mathbf{SSet}^{\mathcal{C}}$$

 Since $P_F$ is a filtered category, and since filtered colimits preserve weak equivalences between simplicial sets, the (ordinary) colimit $\varinjlim D_F$ is a homotopy colimit. Since filtered colimit commute with finite limits and preserve weak equivalences in $\mathbf{SSet}$, and since $P_F$ is a diagram of homotopical functors that preserve finite limits and homotopy limits (corepresentable functors), the functor $F_r : \mathcal{C} \to  \mathbf{SSet}^\circ$ is homotopical and inherits the same preservation properties.
 
 The corresponding diagram of quasicategories
 $$ \mathbf{Ho}_\infty( D_F) : \mathbf{N}(P_F) \to \mathbf{Ho}_\infty(\mathcal{P}(C)) \simeq \mathcal{S}^{ \mathbf{Ho}_\infty(\mathcal{C})}$$
 has the same colimit as the diagram $$ \mathbf{Ho}_\infty(\mathcal{C})_{/F}^{op} \to  \mathbf{Ho}_\infty(\mathcal{C})^{op} \to \mathcal{S}^{ \mathbf{Ho}_\infty(\mathcal{C})}$$
 as the two diagrams are related by the final functor $\mathbf{N} P_F \to  \mathbf{Ho}_\infty(\mathcal{C})_{/F}^{op}$.
 But the colimit of the latter diagram is the covariant presheaf 
$$F :  \mathbf{Ho}_\infty(\mathcal{C}) \to \mathcal{S}$$
we started with, as proved in Corollary 5.3.5.4 of \cite{lurie2009}, so we can conclude.
\end{proof}

\subsection{The general case}

In this subsection, we extend the result obtained in the previous one to the case where $\mathcal{S}$ is replaced by the quasicategory $\mathcal{S}^{\mathbf{Ho}_\infty(\mathcal{D}^{op})} \simeq \mathbf{Ho}_\infty(\mathcal{P}(\mathcal{D}))$. To do this, we rely on a fibred approach and reduce the problem to the unparameterized case.

\begin{proposition}
\label{g_rigidification}
Consider a flat $\infty$-functor $F : \mathbf{Ho}_\infty(\mathcal{C})  \to  \mathbf{Ho}_\infty(\mathcal{P}(\mathcal{D}))$, where $\mathcal{C}$ and $\mathcal{D}$ are fibration categories (here, we really mean that $F$ preserves finite limits).
Then, there exists a homotopical functor $F_r : \mathbf{R}(\mathcal{C}) \to \mathcal{P}(\mathcal{D})$,  such that $\mathbf{Ho}_\infty(F_r)$ is homotopic to $F$ as a $\infty$-functor. Furthermore, $F_r$ preserves finite limits and finite homotopy limits.
\end{proposition}

\begin{proof}
We consider the $\infty$-functor $$F_c : \mathbf{Ho}_\infty(\mathcal{C}) \times \mathbf{Ho}_\infty(\mathcal{D}^{op})  \to \mathcal{S}$$ obtained by transposition from $F$. We may consider a left fibration 
$$\pi_{F_c} : \mathcal{E} \to  \mathbf{Ho}_\infty(\mathcal{C}) \times  \mathbf{Ho}_\infty(\mathcal{D}^{op})$$
corresponding to this functor.

The cocartesian fibration obtained by postcomposing $\pi_{F_c}$ with the projection $$  \mathbf{Ho}_\infty(\mathcal{C}) \times  \mathbf{Ho}_\infty(\mathcal{D}^{op}) \to   \mathbf{Ho}_\infty(\mathcal{D}^{op})$$ can be pulled back along the map $\mathbf{N}(\mathcal{D}^{op}) \to \mathbf{N}_\Delta(L_C(\mathcal{D})^{op}) \simeq  \mathbf{Ho}_\infty(\mathcal{D}^{op})$ obtained from the pseudofunctor $\mathcal{D}^{op} \to L_C(\mathcal{D})^{op}$. The result corresponds, through the higher Grothendieck construction (that is, the straightening functor), to a functor 
$$Q_F : \mathcal{D}^{op} \to \mathbf{QCat}$$
where $\mathbf{QCat}$ is the full subcategory of $\mathbf{SSet}$ spanned by the quasicategories. Actually, we obtain a relative diagrams (and we will only consider relative diagrams indexed by $\mathcal{D}^{op}$), instead of a simplicial diagram indexed by $L_C(\mathcal{D})^{op}$. Note that relying on relative diagrams allows us to work with the same object from the $\infty$-categorical point-of-view.

The  functor $Q_F$ comes equipped with a natural transformation $Q_F \to \Delta_{\mathbf{Ho}_\infty(\mathcal{C})}$, where $$\Delta_{\mathbf{Ho}_\infty(\mathcal{C})} : \mathcal{D}^{op} \to \mathbf{QCat}$$ is the constant functor with value $\mathbf{Ho}_\infty(\mathcal{C})$.

As in the unparameterized case, since $F$ preserves finite limits in the first variable, the quasicategories $Q_F(d)$, for $d$ an object of $\mathcal{D}^{op}$, are cofiltered. This is because they correspond to the domain of the left fibration associated with the $\infty$-functor: $$F_c \circ (id_{\mathbf{Ho}_\infty(\mathcal{C})} \times d) : \mathbf{Ho}_\infty(\mathcal{C}) \to \mathbf{SSet}$$
Here, $d : * \to \mathbf{Ho}_\infty(\mathcal{D}^{op})$ takes the unique object of $*$ to the object $d \in \mathbf{Ho}_\infty(\mathcal{D}^{op})$.

Therefore, we can postcompose the functor $\mathbf{op} \circ Q_F$ (where $\mathbf{op} : \mathbf{Qcat} \to \mathbf{QCat}$ maps a quasicategory to its opposite) with the functor $\rho_\mathbf{Filt}$ provided by \cref{func_filt}, as to obtain a functor
$$P_F : \mathcal{D}^{op} \to \mathbf{FiltPos}$$
taking values in the category of filtered posets. As noted above, this is now only a functor between ordinary categories (i.e, we dropped the simplicial enrichment), but we can consider it as a relative functor in order to keep the information corresponding to the $\infty$-functor we started with. This functor also comes equipped with a natural transformation $\iota_{\mathbf{Pos}} \circ P_F \to \mathbf{op} \circ Q_F$ whose components are final $\infty$-functors (writing $\iota_{\mathbf{Pos}} : \mathbf{Pos} \to \mathbf{QCat}$ for the inclusion). We also have a natural transformation $\mathbf{N} \circ P_F \to \Delta_{\mathbf{Ho}_\infty(\mathcal{C}^{op})}$
where $$\Delta_{ \mathbf{Ho}_\infty(\mathcal{C}^{op})} : \mathcal{D}^{op} \to \mathbf{QCat}$$ is the constant functor with value $\mathbf{Ho}_\infty(\mathcal{C})^{op}$.

The Barwick-Kan model structure on $\mathbf{RelCat}$ being combinatorial, the projective model structure on the functor category $\mathbf{RelCat}^{\mathcal{D}^{op}}$ exists. Since the diagram $Q_F$ obtained by straightening is cofibrant (as the straightening functor is left Quillen and every object is cofibrant in the cocartesian model structure), and since the functor $\xi \circ \rho_{\mathbf{Filt}}$ satisfies the assumption of \cref{lquillen_reedy} by \cref{func_filt}, we can conclude that the diagram $\xi P_F := \xi \circ P_F : \mathcal{D}^{op} \to \mathbf{RelCat}$ is also cofibrant. Applying \cref{fibcat_fibrant_fb}, we can consider a map $P_F \to \Delta_{\mathbf{R}(\mathcal{C})}$ such that its image through $\mathbf{Ho}_\infty$ coincides with the composite: $$\mathbf{Ho}_\infty(\xi P_F) \to \mathbf{Ho}_\infty(\mathbf{op} \circ Q_F) \to \mathbf{Ho}_\infty(\Delta_{\mathbf{R}(\mathcal{C})})$$

We compose back with the functor $\mathbf{op} : \mathbf{Cat} \to \mathbf{Cat}$ and apply the Grothendieck construction. We get an opfibration $$\pi_F : \mathbf{E} \to \mathcal{D}^{op}$$ and a diagram $$\mathbf{y}^{op} \circ D_{0,F} :  \mathbf{E} \to \mathbf{R}(\mathcal{C})^{op} \to \mathcal{P}(\mathbf{R}(\mathcal{C})^{op})$$

We can now compute the (pointwise) left Kan extension of the previous diagram along the opfibration to get a functor $\mathcal{D}^{op} \to \mathbf{SSet}^{\mathbf{R}(\mathcal{C})}$ that transposes to:
$$F_{\mathbf{r},0} : \mathbf{R}(\mathcal{C}) \to \mathbf{SSet}^{\mathcal{D}^{op}}$$

We will now show that $F_{\mathbf{r},0}$ preserves finite limits and maps fibrations (resp. weak equivalences) in $\mathcal{C}$ to pointwise fibrations (resp. weak equivalences) in $\mathbf{SSet}^{\mathcal{D}^{op}}$ (with respect to the Quillen model structure on $\mathbf{SSet}$).

We first prove that $F_{\mathbf{r},0}$ takes value in homotopical diagrams. That is, given an object $c$ in $\mathcal{C}$ and  a weak equivalence $u : d \to d'$ in $\mathcal{D}$, we check that $F_{\mathbf{r},0}(c)$ induces a weak equivalence of simplicial sets: $$F_{\mathbf{r},0}(c)(u) : F_{\mathbf{r},0}(c)(d') \to F_{\mathbf{r},0}(c)(d)$$
But the image of $u$ through the functor $Q_F$ is an equivalence of quasicategories, which, in turn, implies that its image through $P_F$ is a final functor. This means that the canonical map $F_{\mathbf{r},0}(c)(u)$ between the colimits $F_{\mathbf{r},0}(c)(d')$ and $F_{\mathbf{r},0}(c)(d)$ is a weak equivalence of simplicial sets.

 Secondly, since (homotopy) limits in $\mathbf{SSet}^{\mathcal{D}^{op}}$ are computed pointwise, we may check that, for every object $d$ in $\mathcal{D}$, the functor
 $$\mathbf{ev}_d \circ F_{\mathbf{r},0} : P(\mathcal{C}) \to \mathbf{SSet}$$
 preserves finite limits and maps pointwise fibrations (resp. weak equivalence) to fibrations (resp. weak equivalences) for the Quillen model structure on $\mathbf{SSet}$.
 Since the left Kan extension providing $F_{\mathbf{r},0}$ is pointwise, it is enough to check that the colimit
 $$\varinjlim_{x \in \pi_{F,d}} \mathbf{ev}_{\pi_\mathcal{C} x}$$
 satisfies these preservation conditions, where $\pi_{F,d}$ is the fiber of the fibration $\pi_F$ above $d$.
 But $\pi_{F,d}$ coincides with $P_F(d)$ by construction, which is a filtered poset, and all the functors $\mathbf{ev}_{\pi_\mathcal{C} x}$ preserve (finite) limits, pointwise fibrations, and weak equivalences. Therefore, the result follows directly from commutation of finite limits with filtered colimits in $\mathbf{SSet}$, as well as stability of Kan fibrations and weak equivalences of simplicial sets under filtered colimits (see Proposition 3.3 and Theorem 4.1 in \cite{rosicky2009combinatorial}).
 We finally achieve our initial goal by setting $F_{\mathbf{r}} := \mathbf{R}_\mathcal{D^{op}} \circ F_{\mathbf{r},0}$
\end{proof}

\subsection{Rigid left Kan extension}
\label{rlan}

In this subsection, we will construct the desired approximations of homotopy left Kan extensions of exact functors taking values in "presheaves topoi", by means of (spans of) exact functors. 

We consider a diagram 
\[\begin{tikzcd}[ampersand replacement=\&]
	{\mathcal{C}} \&\& {\mathcal{P}(\mathcal{D})} \\
	{\mathcal{C}'}
	\arrow["F", from=1-1, to=1-3]
	\arrow["K"', from=1-1, to=2-1]
\end{tikzcd}\]
in the category $\mathbf{FibCat}$ of fibration categories, where  $\mathcal{P}(\mathcal{D})$ is the "presheaves" fibration category as defined in \cref{p_fibc}. Our goal is to construct a morphism $\mathcal{C'} \to \mathcal{P}(\mathcal{D})$, or at least a span $\mathcal{C'} \leftarrow \mathcal{C}'' \rightarrow \mathcal{P}(\mathcal{D})$ in $\mathbf{FibCat}$, that models the left homotopy Kan extension associated with the corresponding diagram of quasicategories. 

First, we are interested in the left Kan extension along the Yoneda embedding, as any left Kan extension can be constructed from it via the functor $$n_K : c' \mapsto Hom_\mathcal{C}(K(-),c')$$ using the formula: $$Lan_K F \simeq Lan_\mathbf{y} F \circ n_K$$

\subsection{The unparameterized case}

In the quasicategorical context, using the notations introduced in the proof of  \cref{unp_rigidification}, we have the following equivalences, natural in the presheaf $X \in \mathcal{P}(\mathcal{C})$:

\begin{equation}
\begin{aligned}
\label{lan2}
 Hom(Lan_{\mathbf{y}} \mathbf{Ho}_\infty F, X) & \simeq Hom(\mathbf{Ho}_\infty F, X \circ \mathbf{y})\\
 & \simeq Hom(\varinjlim_{x \in \mathbf{Ho}_\infty(\mathcal{C})_{/F}^{op}} \mathbf{y}'(\pi_F x), X \circ \mathbf{y}) \\
 & \simeq \varprojlim_{x \in \mathbf{Ho}_\infty(\mathcal{C})_{/F}} Hom(\mathbf{y}'(\pi_F x), X \circ \mathbf{y}) \\finster
 & \simeq \varprojlim_{x \in \mathbf{Ho}_\infty(\mathcal{C})_{/F}} X(\mathbf{y}(\pi_F x)) \\
 & \simeq \varprojlim_{x \in \mathbf{Ho}_\infty(\mathcal{C})_{/F}} Hom(\mathbf{Y}(\mathbf{y}\pi_F x),X) \\
 & \simeq \varprojlim_{x \in \mathbf{Ho}_\infty(\mathcal{C})_{/F}} Hom(\mathbf{ev}_{\pi_F x},X)\\
 & \simeq Hom(\varinjlim_{x \in \mathbf{Ho}_\infty(\mathcal{C})^{op}_{/F}} \mathbf{ev}_{\pi_F x},X)
\end{aligned}
\end{equation}

This suggests defining our extension $E_F : P(\mathcal{C}) \to \mathbf{SSet}$ as the following filtered colimit of left exact functors,
$$E_F := \varinjlim_{x \in P_F} \mathbf{ev}_{D_{0,F} x}$$
so that the following holds.

\begin{lemma}
\label{unp_lan}
 $\mathbf{Ho}_\infty E_F$ is the left Kan extension of $\mathbf{Ho}_\infty F$ along the Yoneda embedding.
\end{lemma}

\begin{proof}
 There is an induced canonical transformation,
 $$Lan_{\mathbf{y}} \mathbf{Ho}_\infty F \to  \mathbf{Ho}_\infty \varinjlim_{x \in P_F} \mathbf{ev}_{D_{F} x}$$
 obtained from \eqref{lan2} by replacing $\varinjlim_{x \in \mathbf{Ho}_\infty(\mathcal{C})^{op}_{/F}} \mathbf{ev}_{\pi_F x}$ in the last line by the weakly equivalent functor $\varinjlim_{x \in P_F} \mathbf{ev}_{D_{0,F} x}$. This transformation is invertible, thus, using the Yoneda lemma in the appropriate homotopy $2$-category, we obtain the desired result.
\end{proof}

\subsection{The general case}

We first construct an extension $E : P(\mathcal{C}) \to P(\mathcal{D})$, as a parameterized version of the construction we used earlier. To do this, using the notation introduced in the proof of \cref{g_rigidification}, we compute the left Kan extension of the diagram
$$\mathbf{Y}' \circ \mathbf{y}^{op} \circ D_{0,F} : \mathbf{E} \to \mathbf{R}(\mathcal{C})^{op} \to \mathcal{P}(\mathbf{R}(\mathcal{C}))^{op} \to \mathbf{SSet}^{\mathcal{P}(\mathcal{C})}$$
where $\mathbf{Y}'$ maps a simplicial presheaf $X \in \mathcal{P}(\mathbf{R}(\mathcal{C}))$ to the corepresentable functor $Hom_{\mathcal{P}(\mathcal{C})}(\mathbf{y}^*X,-)$, along the opfibration $\mathbf{E} \to \mathcal{D}^{op}$.

Transposing, we get a map $\mathcal{P}(\mathcal{C}) \to \mathbf{SSet}^{\mathcal{D}^{op}}$, which we postcompose with  $\mathbf{R}_\mathcal{D}$ to obtain the extension $E : P(\mathcal{C}) \to P(\mathcal{D})$.

We achieve our goal by precomposing with the functor $$n_K : \mathcal{C} \to \mathcal{P}(\mathcal{C})$$ defined by the mapping $$c' \mapsto \mathbf{R}_\mathcal{C}\mathbf{N}(H'(K-,c'))$$ where $H'$ has been defined in \cref{func_homf}, and by considering the following variation on \cref{p_fibc}:

\[\begin{tikzcd}[ampersand replacement=\&]
	{\mathcal{Q}'(K)} \&\& {P\mathcal{P}(\mathcal{C})} \\
	\\
	{\mathcal{C}' \times \mathcal{P}(\mathcal{C})} \&\& {\mathcal{P}(\mathcal{C}) \times \mathcal{P}(\mathcal{C})}
	\arrow[from=1-1, to=1-3]
	\arrow[from=1-1, to=3-1]
	\arrow["\lrcorner"{anchor=center, pos=0.125}, draw=none, from=1-1, to=3-3]
	\arrow["{<\pi_0,\pi_1>}", from=1-3, to=3-3]
	\arrow["{n_K \times id_{\mathcal{P}(\mathcal{C})}}"', from=3-1, to=3-3]
\end{tikzcd}\]

To sum up, we have the following result:
\begin{proposition}
\label{rlan_c}
 The functor $E$ is a morphism of fibration categories such that $\mathbf{Ho}_\infty(E \circ n_K)$ is the left Kan extension of $\mathbf{Ho}_\infty(F)$ along $\mathbf{Ho}_\infty(K)$:
\[\begin{tikzcd}[ampersand replacement=\&]
	{\mathbf{Ho}_\infty(\mathcal{C})} \&\& {\mathbf{Ho}_\infty(\mathcal{P}(\mathcal{D}))} \\
	{\mathbf{Ho}_\infty(\mathcal{C}')}
	\arrow[""{name=0, anchor=center, inner sep=0}, "{\mathbf{Ho}_\infty(F)}", from=1-1, to=1-3]
	\arrow["{\mathbf{Ho}_\infty(K)}"', from=1-1, to=2-1]
	\arrow[""{name=1, anchor=center, inner sep=0}, "{\mathbf{Ho}_\infty(E \circ n_K)}"', from=2-1, to=1-3]
	\arrow[shift right=5, shorten <=2pt, shorten >=2pt, Rightarrow, from=0, to=1]
\end{tikzcd}\] 
\end{proposition}

\begin{remark}
 The homotopy left Kan extension of $F$ along $K$ is hence represented by the following span of exact functors:

\[\begin{tikzcd}[ampersand replacement=\&]
	\&\& {\mathcal{Q}'(K)} \\
	\&\&\& {\mathcal{P}(\mathcal{C})} \\
	{\mathcal{C}'} \&\&\&\& {\mathcal{P}(\mathcal{D})}
	\arrow[from=1-3, to=2-4]
	\arrow[from=1-3, to=3-1]
	\arrow["E", from=2-4, to=3-5]
\end{tikzcd}\]
\end{remark}

\section{Applications}

In this section, we construct a well-behaved implementation of the localization functor: $$\mathbf{Ho}_\infty : \mathbf{FibCat} \to \mathbf{QCat}_{lex}$$
From there, we will provide a new proof that this functor is a DK-equivalence, relying a slight variation on Cisinski's criterion characterizing the DK-equivalence between fibration categories.

\begin{definition}
 A functor $F : \mathcal{C} \to \mathcal{D}$ between fibration categories is weakly exact when it satisfies the assumption of \cref{corr}: it is homotopical, it maps pullbacks along fibrations to pullbacks that are homotopy pullbacks, and it preserves the terminal object up to weak equivalence (i.e $F(*_\mathcal{C}) \to *_\mathcal{D}$ is a weak equivalence).
\end{definition}

\begin{lemma}
 The relative nerve functor $$\mathbf{N}_\xi : \mathbf{RelCat} \to \mathbf{sSSet}$$ of \cite[Theorem 6.1]{barwick2012relative}, from the model category of relative categories to the model structure of complete Segal spaces on bisimplicial sets,t induces a weakly exact functor $$\mathbf{Ho}_\infty : \mathbf{FibCat} \to \mathbf{QCat}_{lex}$$ fitting in a commutative square

\[\begin{tikzcd}[ampersand replacement=\&]
	{\mathbf{FibCat}} \&\& {\mathbf{QCat}_{lex}} \\
	\\
	{\mathbf{RelCat}} \& { \mathbf{sSSet}} \& {\mathbf{QCat}}
	\arrow["{\mathbf{Ho}_\infty}", from=1-1, to=1-3]
	\arrow[from=1-1, to=3-1]
	\arrow[from=1-3, to=3-3]
	\arrow["{\mathbf{N}_\xi}"', from=3-1, to=3-2]
	\arrow["{i_1^*}"', from=3-2, to=3-3]
\end{tikzcd}\]
where $i_1^*$ is the right adjoint in a Quillen equivalence between complete Segal spaces and quasicategories (see \cite[Theorem 4.11]{joyal2007quasi}).
\end{lemma}

\begin{proof}
 Because fibration categories are fibrant object of $\mathbf{RelCat}$, as proved in \cite{meier2016fibration}, and because $\mathbf{N}_\xi$ and $i_1^*$ are right Quillen functors, the composite takes values in fibrant objects for the Joyal model structure, i.e. quasicategories.
 That the functor restricted to the category $\mathbf{FibCat}$ of fibration categories factors moreover through the subcategory $\mathbf{QCat}_{lex}$ of $\mathbf{QCat}$ is clear since fibration categories underlie finitely complete $(\infty,1)$-categories. The functor is moreover homotopical by Ken Brown's lemma. 
 
 The underlying $\infty$-functor from $\mathbf{RelCat}$ to $\mathbf{QCat}$ is an equivalence and a right adjoint, in particular it preserves limits and homotopy limits. Since pullbacks along fibrations in $\mathbf{FibCat}$ are homotopy limits, $\mathbf{Ho}_\infty$ takes them to pullback that are homotopy pullbacks in $\mathbf{QCat}$ (hence in $\mathbf{QCat}_{lex}$). Similarly, the terminal object is preserved in the higher sense (and here even in the $1$-categorical sense), so we can conclude.
\end{proof}

We now look to extend to following key result characterizing DK-equivalences among exact functors:

\begin{theorem}[Cisinski]
\label{cisinski_dke}
 Given fibration categories $\mathcal{F}_0$ and $\mathcal{F}_1$, as well as an exact functor $H : \mathcal{F}_0 \to \mathcal{F}_1$, the following are equivalent:
 \begin{itemize}
  \item $H$ is a DK-equivalence.
  \item $\mathbf{Ho}(H) : \mathbf{Ho}(\mathcal{F}_0) \to \mathbf{Ho}(\mathcal{F}_1)$ is an equivalence of categories.
  \item $H$ satisfies the following two \textit{approximations properties}:
  \begin{enumerate}
   \item[(AP1)] $H$ reflects weak equivalences.
   \item[(AP2)] For every objects $x_0 \in \mathcal{F}_0$ and $y_1 \in \mathcal{F}_1$, and every morphism $y_1 \to H(x_0)$ in $\mathcal{F}_1$, there exists a commutative square in $\mathcal{F}_1$,

\[\begin{tikzcd}[ampersand replacement=\&]
	{y_1} \&\& {H(x_0)} \\
	\\
	{y'_1} \&\& {H(y_0)}
	\arrow[from=1-1, to=1-3]
	\arrow["\sim"{description}, from=3-1, to=1-1]
	\arrow["\sim"{description}, from=3-1, to=3-3]
	\arrow["{H(f)}"', from=3-3, to=1-3]
\end{tikzcd}\]
   with $f : y_0 \to x_0$ an arrow in $\mathcal{F}_0$, and where the indicated arrows are weak equivalences.
  \end{enumerate}
 \end{itemize}
\end{theorem}

\begin{proposition}
\label{corr_r}
 Let $F: \mathcal{C} \to \mathcal{D}$ be a weakly exact functor between fibration categories. Assume that $P\mathcal{D}$ is a path-object for $\mathcal{D}$.
 Then $F$ is a DK-equivalence if it satisfies Cisinski's approximation properties of \cref{cisinski_dke}.
\end{proposition}

\begin{proof}
 Applying \cref{corr}, we have a pullback diagram

\[\begin{tikzcd}[ampersand replacement=\&]
	{\mathcal{C}'} \&\& {P\mathcal{D}} \\
	\\
	{\mathcal{C} \times \mathcal{D}} \&\& {\mathcal{D} \times \mathcal{D}}
	\arrow["u", from=1-1, to=1-3]
	\arrow["{<\alpha,\beta>}"', from=1-1, to=3-1]
	\arrow["\lrcorner"{anchor=center, pos=0.125}, draw=none, from=1-1, to=3-3]
	\arrow["{<\partial_0,\partial_1>}", from=1-3, to=3-3]
	\arrow["{F \times id_{\mathcal{D}}}"', from=3-1, to=3-3]
\end{tikzcd}\]
where $\alpha$ and $\beta$ are exact functors.

Suppose that $F$ satisfies the approximation properties. By definition of the weak equivalences in $\mathcal{C}'$, since $F$ satisfies (AP1), so does $\beta$. For (AP2), consider a morphism $d \to \beta(c')$ in $\mathcal{D}$. We have a diagram 
\[\begin{tikzcd}[ampersand replacement=\&]
	d \&\& {\beta(c')} \\
	\\
	{D'} \&\& {C'} \\
	\\
	\&\& {F\alpha (c')}
	\arrow[from=1-1, to=1-3]
	\arrow[dashed, from=3-1, to=1-1]
	\arrow["\ulcorner"{anchor=center, pos=0.125, rotate=90}, draw=none, from=3-1, to=1-3]
	\arrow[dashed, from=3-1, to=3-3]
	\arrow["\sim"{description}, two heads, from=3-3, to=1-3]
	\arrow["\sim"{description}, two heads, from=3-3, to=5-3]
\end{tikzcd}\]
where we constructed the pullback $D'$. Using that $F$ has the approximation property (AP2), we may complete the previous diagram as follows:
\[\begin{tikzcd}[ampersand replacement=\&]
	\&\& d \&\& {\beta(c')} \\
	\\
	\&\& {D'} \&\& {C'} \\
	\& {D''} \\
	{F(c)} \&\&\&\& {F\alpha (c')}
	\arrow[from=1-3, to=1-5]
	\arrow[dashed, from=3-3, to=1-3]
	\arrow["\ulcorner"{anchor=center, pos=0.125, rotate=90}, draw=none, from=3-3, to=1-5]
	\arrow[dashed, from=3-3, to=3-5]
	\arrow["\sim"{description}, two heads, from=3-5, to=1-5]
	\arrow["\sim"{description}, two heads, from=3-5, to=5-5]
	\arrow["\sim"{description}, from=4-2, to=3-3]
	\arrow["\sim"{description}, from=4-2, to=5-1]
	\arrow[from=5-1, to=5-5]
\end{tikzcd}\]
and finally form the weak equivalence/fibration factorization as in the following diagram
\[\begin{tikzcd}[ampersand replacement=\&]
	{D''} \\
	\& {D'''} \\
	\&\& P \&\& {C'} \\
	\\
	\&\& {F(c) \times d} \&\& {F\alpha (c') \times \beta(c')}
	\arrow["\sim"{description}, from=1-1, to=2-2]
	\arrow[curve={height=-12pt}, from=1-1, to=3-5]
	\arrow[curve={height=12pt}, from=1-1, to=5-3]
	\arrow[two heads, from=2-2, to=3-3]
	\arrow[from=3-3, to=3-5]
	\arrow[two heads, from=3-3, to=5-3]
	\arrow["\ulcorner"{anchor=center, pos=0.125}, draw=none, from=3-3, to=5-5]
	\arrow[two heads, from=3-5, to=5-5]
	\arrow[from=5-3, to=5-5]
\end{tikzcd}\]
where $D''' \to F(c) \times d$ is a span of trivial fibrations since $D'' \to  F(c) \times d$ is a span of weak equivalences by construction. This defines an object $d'$ of $\mathcal{C}'$ and a map $f' :d' \to c'$, and we have the following commutative diagram,
\[\begin{tikzcd}[ampersand replacement=\&]
	d \&\& {\beta(c')} \\
	\\
	{ d} \&\& {\beta(d') (=d)}
	\arrow[from=1-1, to=1-3]
	\arrow["{id_d}", from=3-1, to=1-1]
	\arrow["{id_d}"', from=3-1, to=3-3]
	\arrow["{\beta(f')}"', from=3-3, to=1-3]
\end{tikzcd}\]
proving that $\beta$ has the approximation property.

Since $\beta$ is an exact functor between fibration categories, it is therefore a DK-equivalence by \cref{cisinski_dke}. Using the second statement of \cref{corr}, and the $2$-out-of-$3$ property for DK-equivalences, we conclude that $F$ is also a DK-equivalence.
\end{proof}

\begin{remark}
\label{rmk:corr_r}
 In the proof of \cref{corr_r}, we only need $\beta$ to satisfy the approximation property, which may be the case even if $F$ does not satisfy it. Therefore, the assumption of \cref{corr_r} could be weakened accordingly.
\end{remark}

We can now give an alternative proof of \cite[Theorem 4.9]{szumilo2014two}:

\begin{theorem}
 The functor $$\mathbf{Ho}_\infty : \mathbf{FibCat} \to \mathbf{QCat}_{lex}$$
 is a DK-equivalence.
\end{theorem}

\begin{proof}
 Since both $\mathbf{FibCat}$ and $\mathbf{QCat}_{lex}$ enjoy the structure of a fibration category, and since the functor $\mathbf{Ho}_\infty : \mathbf{FibCat} \to \mathbf{QCat}_{lex}$ is weakly exact, we can make use of \cref{corr_r}. Actually, we will rely on \cref{rmk:corr_r}, and check the approximation property for the functor $\beta$ defined from the following pullback:

\[\begin{tikzcd}[ampersand replacement=\&]
	{\mathbf{FQ}} \&\& {P\mathbf{QCat}_{lex}} \\
	\\
	{\mathbf{FibCat} \times \mathbf{QCat}_{lex}} \&\& {\mathbf{QCat}_{lex} \times \mathbf{QCat}_{lex}}
	\arrow["u", from=1-1, to=1-3]
	\arrow["{<\alpha,\beta>}"', from=1-1, to=3-1]
	\arrow["\lrcorner"{anchor=center, pos=0.125}, draw=none, from=1-1, to=3-3]
	\arrow["{<\partial_0,\partial_1>}", from=1-3, to=3-3]
	\arrow["{\mathbf{Ho}_\infty \times id_{\mathbf{QCat}_{lex}}}"', from=3-1, to=3-3]
\end{tikzcd}\] 
 That $\beta$ reflects weak equivalence follows from $\mathbf{Ho}_\infty$ reflecting weak equivalences: since the weak equivalences in $\mathbf{FibCat}$ are the exact functors that are DK-equivalences, this boils down to the very definition of a DK-equivalence.
 
 To check (AP2), consider an $\infty$-functor $\mathcal{D}_0 \to \beta(c')$ in $\mathbf{QCat}_{lex}$, where $c'$ is an object of $\mathbf{FQ}$ consisting of a span $\mathcal{C}' \to \mathbf{Ho}_\infty \alpha(c') \times \beta(c')$. Form the pullback $\mathcal{D}$ in the diagram below, and observe that the finitely complete quasicategory $\mathcal{D}$ may be presented by the tribe $\overline{\mathcal{D}}$ of essentially representable fibrant simplicial presheaves on $\mathfrak{C}(\mathcal{D})$, as argued in Chapter 1 (in a more complex setting).
 
\[\begin{tikzcd}[ampersand replacement=\&]
	\&\& {\mathcal{D}_0} \&\& {\beta(c')} \\
	\\
	{\mathbf{Ho}_\infty(\mathbf{R}(\overline{\mathcal{D}}))} \&\& {\mathcal{D}} \&\& {\mathcal{C}'} \\
	\\
	\&\&\&\& {\mathbf{Ho}_\infty\alpha (c')}
	\arrow[from=1-3, to=1-5]
	\arrow["\sim"{description}, from=3-1, to=3-3]
	\arrow[dashed, from=3-3, to=1-3]
	\arrow["\ulcorner"{anchor=center, pos=0.125, rotate=90}, draw=none, from=3-3, to=1-5]
	\arrow[dashed, from=3-3, to=3-5]
	\arrow["\sim"{description}, two heads, from=3-5, to=1-5]
	\arrow["\sim"{description}, two heads, from=3-5, to=5-5]
\end{tikzcd}\]
 Here, $\mathbf{Ho}_\infty(\mathbf{R}(\overline{\mathcal{D}})) \to \mathcal{D}$ is a fixed equivalence of quasicategories.

  Applying \cref{g_rigidification} to the corresponding $\infty$-functor $F' : \mathbf{Ho}_\infty(\overline{\mathcal{D}}) \to \mathbf{Ho}_\infty \alpha(c')$, we get an exact functor $F_r : \mathbf{R}(\overline{\mathcal{C}}) \to \alpha(c')$ such that $\mathbf{Ho}_\infty(F_r)$ is homotopic to $F'$ (modulo the fixed equivalence $\mathbf{Ho}_\infty(\mathbf{R}(\overline{\mathcal{D}})) \simeq \mathcal{D}$. Consider the cylinder object $J \times \mathbf{Ho}_\infty(\mathbf{R}(\overline{\mathcal{D}}))$ for $\mathbf{Ho}_\infty(\mathbf{R}(\overline{\mathcal{D}}))$, where $J$ is the nerve of the free isomorphism category, and consider a homotopy $H$ between the two previous $\infty$-functors $\mathbf{Ho}_\infty(\overline{\mathcal{D}}) \to \mathbf{Ho}_\infty \alpha(c')$ as in the diagram below.

\[\begin{tikzcd}[ampersand replacement=\&]
	\&\&\& {\mathcal{D}_0} \&\& {\beta(c')} \\
	\\
	\& {\mathbf{Ho}_\infty(\mathbf{R}(\overline{\mathcal{D}}))} \&\& {\mathcal{D}} \&\& {\mathcal{C}'} \\
	{J \times \mathbf{Ho}_\infty(\mathbf{R}(\overline{\mathcal{D}}))} \\
	\& {\mathbf{Ho}_\infty(\mathbf{R}(\overline{\mathcal{D}}))} \&\&\&\& {\mathbf{Ho}_\infty\alpha (c')}
	\arrow[from=1-4, to=1-6]
	\arrow["\sim"{description}, from=3-2, to=3-4]
	\arrow["\sim"{description}, from=3-2, to=4-1]
	\arrow[dashed, from=3-4, to=1-4]
	\arrow["\ulcorner"{anchor=center, pos=0.125, rotate=90}, draw=none, from=3-4, to=1-6]
	\arrow[dashed, from=3-4, to=3-6]
	\arrow["\sim"{description}, two heads, from=3-6, to=1-6]
	\arrow["\sim"{description}, two heads, from=3-6, to=5-6]
	\arrow["H", from=4-1, to=5-6]
	\arrow["\sim"{description}, from=5-2, to=4-1]
	\arrow[from=5-2, to=5-6]
\end{tikzcd}\]

We now can form the following diagram, 

\adjustbox{scale=0.8}{
\begin{tikzcd}[ampersand replacement=\&]
	\& {\mathbf{Ho}_\infty(\mathbf{R}(\overline{\mathcal{D}}))} \\
	{\mathcal{D}'''} \&\& {\mathcal{D}''} \\
	\&\&\& {\mathcal{D}'} \&\& {\mathcal{C}'} \\
	\\
	\& {\mathbf{Ho}_\infty(\mathbf{R}(\overline{\mathcal{D}})) \times \mathcal{D}_0} \&\& {J \times \mathbf{Ho}_\infty(\mathbf{R}(\overline{\mathcal{D}})) \times \mathcal{D}_0} \&\& {\mathbf{Ho}_\infty(\alpha (c')) \times \beta(c')}
	\arrow["\sim"{description}, from=1-2, to=2-3]
	\arrow[curve={height=-6pt}, from=1-2, to=3-6]
	\arrow["{<i_1,w>}"', curve={height=6pt}, from=1-2, to=5-4]
	\arrow["\sim"{description}, from=2-1, to=2-3]
	\arrow[two heads, from=2-1, to=5-2]
	\arrow["\ulcorner"{anchor=center, pos=0.125}, draw=none, from=2-1, to=5-4]
	\arrow[two heads, from=2-3, to=3-4]
	\arrow[from=3-4, to=3-6]
	\arrow[two heads, from=3-4, to=5-4]
	\arrow["\ulcorner"{anchor=center, pos=0.125}, draw=none, from=3-4, to=5-6]
	\arrow[two heads, from=3-6, to=5-6]
	\arrow["{i_0 \times id_{\mathcal{D}_0}}"', from=5-2, to=5-4]
	\arrow[from=5-4, to=5-6]
\end{tikzcd}
}
where $\mathcal{D}'$ is obtained by pullback, $w$ is the map from $\mathbf{Ho}_\infty(\mathbf{R}(\overline{\mathcal{D}}))$ to $\mathcal{D}_0$, and where $\mathcal{D}''$ is obtained by factoring $\mathbf{Ho}_\infty(\mathbf{R}(\overline{\mathcal{D}})) \to \mathcal{D}'$ as a weak equivalence followed by a fibration. By construction $\mathcal{D}'' \to J \times \mathbf{Ho}_\infty(\mathbf{R}(\overline{\mathcal{D}})) \times \mathcal{D}_0$ is span of trivial fibration that moreover defines an object of $\mathbf{FQ}$. Since $\mathcal{D}''' \to \mathcal{D}''$ is a weak equivalence (as the pullback of a weak equivalence along a fibration), the span $\mathcal{D}''' \to \mathbf{Ho}_\infty(\mathbf{R}(\overline{\mathcal{D}})) \times \mathcal{D}_0$ also defines an object $d'$ fo $\mathbf{FQ}$. Moreover, we have obtained a morphism $f' : d' \to c'$ in $\mathbf{FQ}$, providing us with a square 
\[\begin{tikzcd}[ampersand replacement=\&]
	{\mathcal{D}_0} \&\& {\beta(c')} \\
	\\
	{\mathcal{D}_0} \&\& {\beta(d') (=\mathcal{D}_0)}
	\arrow[from=1-1, to=1-3]
	\arrow["{id_{\mathcal{D}_0}}", from=3-1, to=1-1]
	\arrow["{id_{\mathcal{D}_0}}"', from=3-1, to=3-3]
	\arrow["{\beta(f')}"', from=3-3, to=1-3]
\end{tikzcd}\]
which proves that (AP2) holds.

This concludes the proof that $\mathbf{Ho}_\infty : \mathbf{FibCat} \to \mathbf{QCat}_{lex}$ is a DK-equivalence. 
\end{proof}

\section{Extension to tribes}

In this last section, we recall the definition of a tribe (as introduced in \cite{joyal2017notes}). Our goal is, then, to extend the results formulated in terms of fibration categories earlier in this document to the context of tribes. 
The notion of tribes is indeed closely related to fibration categories:

\begin{definition}
 A  tribe is a category $\mathcal{T}$ equipped with a class of morphism $\mathbf{F}$, the fibrations, that are stable under composition, contains any isomorphisms, and such that:
 
 \begin{itemize}
  \item $\mathcal{T}$ admits a terminal object $*$, and the unique map $x \to *$ is a fibration for every object $x$.
  \item $\mathcal{T}$ admits pullbacks along fibrations, and the base change of a fibration is a fibration.
  \item Every morphism factors as an anodyne map (i.e, a map that has the left lifting property against fibrations) followed by a fibration.
  \item Anodyne maps are stable under pullbacks along fibrations.
 \end{itemize}

A functor $P : \mathcal{T} \to \mathcal{T}'$ between tribes is a morphism of tribes when it preserves the corresponding structure: it maps fibrations (resp. anodyne maps) to fibrations (resp. anodyne maps), and preserves the terminal object as well as pullbacks along fibrations. 
\end{definition}

\begin{lemma}
\label{corr_tribe}
 Let $F : \mathcal{T} \to \mathcal{S}$ be a functor between tribes. Assume that $F$ maps pullbacks along fibrations to pullback squares that are homotopy pullbacks, that $F(*_\mathcal{T}) \to *_\mathcal{S}$ is a weak equivalence and that anodyne morphisms are preserved by $F$. Then, in the following pullback, 
\[\begin{tikzcd}[ampersand replacement=\&]
	{\mathcal{T}'} \&\& {P\mathcal{S}} \\
	\\
	{\mathcal{T} \times \mathcal{S}} \&\& {\mathcal{S} \times \mathcal{S}}
	\arrow["u", from=1-1, to=1-3]
	\arrow["{<\alpha,\beta>}"', from=1-1, to=3-1]
	\arrow["\lrcorner"{anchor=center, pos=0.125}, draw=none, from=1-1, to=3-3]
	\arrow["{<\partial_0,\partial_1>}", from=1-3, to=3-3]
	\arrow["{F \times id_{\mathcal{S}}}"', from=3-1, to=3-3]
\end{tikzcd}\]
computed in $\mathbf{Cat}$, $\mathcal{T}'$ inherits a tribe structure structure and the two projections $\alpha$ and $\beta$ are morphisms of tribes.

Moreover, if $P\mathcal{S}$ is a path-object for $\mathcal{S}$, with a map $\iota  : \mathcal{S} \to P\mathcal{S}$, then we have the following diagram of fibration categories
\[\begin{tikzcd}[ampersand replacement=\&]
	\& {\mathcal{T}} \\
	{\mathcal{T}} \&\& { \mathcal{S}} \\
	\& {\mathcal{T}'}
	\arrow["{id_\mathcal{T}}"', from=1-2, to=2-1]
	\arrow["F", from=1-2, to=2-3]
	\arrow["m"{description}, dashed, from=1-2, to=3-2]
	\arrow["\alpha", from=3-2, to=2-1]
	\arrow["\beta"', from=3-2, to=2-3]
\end{tikzcd}\]
where $m : \mathcal{T} \to \mathcal{T}'$ is a weak equivalence.
\end{lemma}

\begin{proof}
The notion of fibration is defined as in \cref{corr} and enjoys the same properties.

Since $F$ preserves anodyne maps and pullbacks along fibrations, it maps weak equivalences (i.e homotopy equivalence in $\mathcal{T}$) to weak equivalences.

 Furthermore, we must check that every morphism $f : x \to y$ factors as a componentwise anodyne morphism, where by the term "components" refers to the following maps,
\begin{align*}
 x_\mathcal{T} \to z_\mathcal{T} \to y_\mathcal{T} \\
 x_\mathcal{S} \to z_\mathcal{S} \to y_\mathcal{S} \\
 X \to Z \to  Y
\end{align*}
which is followed by a fibration. To do so, we first form the factorization of $\alpha(f)$ and $\beta(f)$ in $\mathcal{S}$ (through some object $z_\mathcal{S}$) and $\mathcal{T}$ (through $z_\mathcal{T}$), then we extend the factorization as follows:
\[\begin{tikzcd}[ampersand replacement=\&]
	X \& Z \& {P_z} \&\& Y \\
	\\
	\\
	{F(x_\mathcal{T}) \times x_\mathcal{S}} \&\& {F(z_\mathcal{T}) \times z_\mathcal{S}} \&\& {F(y_\mathcal{T}) \times y_\mathcal{S}}
	\arrow["\sim"{description}, from=1-1, to=1-2]
	\arrow[from=1-1, to=4-1]
	\arrow[two heads, from=1-2, to=1-3]
	\arrow[from=1-3, to=1-5]
	\arrow[from=1-3, to=4-3]
	\arrow["\lrcorner"{anchor=center, pos=0.125}, draw=none, from=1-3, to=4-5]
	\arrow[from=1-5, to=4-5]
	\arrow[from=4-1, to=4-3]
	\arrow[from=4-3, to=4-5]
\end{tikzcd}\]
where, in order to make the observation that $Z \to F(z_\mathcal{T}) \times z_\mathcal{S}$ indeed gives a span of trivial fibrations, we crucially rely on the fact that $F$ preserves weak equivalences.

We claim that the so-called componentwise anodyne morphisms are precisely the anodyne morphisms. Indeed, if $f :x \to y$ is a componentwise anodyne morphism, and $p : a \to b$ is a fibration, we may solve a lifting problem of $f$ against $p$ by first doing so with the components in $\mathcal{S}$ and $\mathcal{T}$ (that is, taking the image under $\alpha$ and $\beta$), then by extending the lifts to $P\mathcal{S}$ as follows (using that $X  \to Y$ is anodyne by assumption):

\vspace{2em}
\adjustbox{scale=0.9}{
\begin{tikzcd}[ampersand replacement=\&]
	\& X \&\&\&\& A \\
	\&\&\&\& \bullet \\
	Y \&\&\& B \\
	\& {F(x_\mathcal{T}) \times x_\mathcal{S}} \&\&\&\& {F(a_\mathcal{T}) \times a_\mathcal{S}} \\
	\\
	{F(y_\mathcal{T}) \times y_\mathcal{S}} \&\&\& {F(b_\mathcal{T}) \times b_\mathcal{S}}
	\arrow[curve={height=-12pt}, from=1-2, to=1-6]
	\arrow[from=1-2, to=3-1]
	\arrow[from=1-2, to=4-2]
	\arrow[two heads, from=1-6, to=2-5]
	\arrow[from=1-6, to=4-6]
	\arrow[from=2-5, to=3-4]
	\arrow[from=2-5, to=4-6]
	\arrow["\lrcorner"{anchor=center, pos=0.125, rotate=-45}, draw=none, from=2-5, to=6-4]
	\arrow["H"{description}, curve={height=-18pt}, dashed, from=3-1, to=1-6]
	\arrow["{H_0}"{description}, curve={height=-12pt}, from=3-1, to=2-5]
	\arrow[curve={height=-12pt}, from=3-1, to=3-4]
	\arrow[from=3-1, to=6-1]
	\arrow[from=3-4, to=6-4]
	\arrow[from=4-2, to=4-6]
	\arrow[from=4-2, to=6-1]
	\arrow[from=4-6, to=6-4]
	\arrow["{F(h_\mathcal{T}) \times h_\mathcal{S}}"{description}, from=6-1, to=4-6]
	\arrow[from=6-1, to=6-4]
\end{tikzcd}}
 
\vspace{2em} 
Conversely, an anodyne morphism $f : x  \to y$ may be factored as a componentwise anodyne morphism $f' : x \to x'$ followed by a fibration $x' \to y$, and we get a lift as below:

\[\begin{tikzcd}[ampersand replacement=\&]
	x \&\& {x'} \\
	\\
	y \&\& y
	\arrow[from=1-1, to=1-3]
	\arrow["\sim"{description}, from=1-1, to=3-1]
	\arrow[two heads, from=1-3, to=3-3]
	\arrow[dashed, from=3-1, to=1-3]
	\arrow[from=3-1, to=3-3]
\end{tikzcd}\]
This lift exhibits $f$ as a retract of $f'$, hence the components of $f$ are retracts of the components of $f'$, which means that $f$ is also a componentwise anodyne morphism.

Using the construction of the pullbacks in $\mathcal{T}'$ discussed in the proof of \cref{corr}, together with the fact that $F$ preserves anodyne morphisms, we see that anodyne morphisms are also stable under pullback along fibrations. This concludes the proof that $\mathcal{T}'$ is a tribe. By construction, the projections $\alpha$ and $\beta$ are morphisms of tribes.

The rest of the statement is obtained just as in the proof of \cref{corr}.
\end{proof}

\subsection{The Yoneda embedding}

In this subsection, we will specialize the results of \cref{fibcat_ye} by substituting the fibration category $\mathcal{C}$ for a tribe $\mathcal{T}$.

The fibration category $\mathcal{P}(\mathcal{T})$ introduced in the context of fibration categories is not just a fibration category, but even a ($\pi$-)tribe: this is because the corresponding model structure is proper, as proved in Théorème 3.2 of \cite{cisinski2010invariance}, and because the cofibrations are the monomorphisms. We also have the structure of a tribe on $\mathbf{R}(\mathcal{C})$.

Applying \cref{corr_tribe} in place of \cref{corr}, we can restate the results of \cref{fibcat_yoneda} about the Yoneda embedding, using the same notations.

\begin{lemma}[The Yoneda embedding for tribes]
\label{tribe_yoneda}
 $\mathcal{Q}(\mathcal{T})$ is a tribe, and the projections $\mathcal{Q}(\mathcal{T}) \to \mathcal{T} \times \mathcal{P}(\mathcal{T})$ are morphisms of tribes, and similarly for $\mathbf{Q}(\mathcal{T})$.

\end{lemma}

\subsection{Rigidifying flat $\infty$-functors}

Finally, we restate the results of \cref{section3} in terms of tribes. We first have the following analogue of \cref{g_rigidification}. 

\begin{proposition}
\label{g_rigidification_tribe}
Consider a flat $\infty$-functor $F : \mathbf{Ho}_\infty(\mathcal{T})  \to  \mathbf{Ho}_\infty(\mathcal{P}(\mathcal{S}))$, where $\mathcal{T}$ and $\mathcal{S}$ are tribes.
Then, there exists a homotopical functor $F_r : \mathbf{R}(\mathcal{C}) \to \mathcal{P}(\mathcal{D})$,  such that $\mathbf{Ho}_\infty(F_r)$ is homotopic to $F$ as a $\infty$-functor. Moreover $F_r$ preserves finite limits and finite homotopy limits, and maps anodyne morphisms to anodyne morphisms.
\end{proposition}

\begin{proof}
 The construction is just like the fibration category case. To deduce the additional property for $F_r$, we just observe that the anodyne morphisms in $\mathcal{P}(\mathcal{D})$ are precisely the trivial cofibration (between fibrant objects), namely the monomorphisms that are also weak equivalences. We already know that $F_r$ preserves weak equivalences, it also preserves monomorphisms because it preserves finite limits. 
\end{proof}

We now consider a diagram 
\[\begin{tikzcd}[ampersand replacement=\&]
	{\mathcal{T}} \&\& {\mathcal{P}(\mathcal{S})} \\
	{\mathcal{T}'}
	\arrow["F", from=1-1, to=1-3]
	\arrow["K"', from=1-1, to=2-1]
\end{tikzcd}\]
where $\mathcal{T}$, $\mathcal{T}'$ and $\mathcal{S}$ are tribes.
Arguing as for \cref{rlan}, but making use of \cref{tribe_yoneda} and \cref{g_rigidification_tribe} instead of their original counterpart for fibration categories, we have the following restatement of \cref{rlan_c} in terms of tribes, keeping the notation introduced there.

\begin{proposition}
\label{rlan_c_tribe}
 The functor $E$ is a morphism of tribes such that $\mathbf{Ho}_\infty(E \circ n_K)$ is the left Kan extension of $\mathbf{Ho}_\infty(F)$ along $\mathbf{Ho}_\infty(K)$:
\[\begin{tikzcd}[ampersand replacement=\&]
	{\mathbf{Ho}_\infty(\mathcal{T})} \&\& {\mathbf{Ho}_\infty(\mathcal{P}(\mathcal{S}))} \\
	{\mathbf{Ho}_\infty(\mathcal{T}')}
	\arrow[""{name=0, anchor=center, inner sep=0}, "{\mathbf{Ho}_\infty(F)}", from=1-1, to=1-3]
	\arrow["{\mathbf{Ho}_\infty(K)}"', from=1-1, to=2-1]
	\arrow[""{name=1, anchor=center, inner sep=0}, "{\mathbf{Ho}_\infty(E \circ n_K)}"', from=2-1, to=1-3]
	\arrow[shift right=5, between={0.2}{0.8}, Rightarrow, from=0, to=1]
\end{tikzcd}\]
\end{proposition}

\begin{remark}
 The homotopy left Kan extension of $F$ along $K$ is hence represented by the following span of morphisms of tribes:
 
\[\begin{tikzcd}[ampersand replacement=\&]
	\&\& {\mathcal{Q}'(K)} \\
	\&\&\& {\mathcal{P}(\mathcal{T})} \\
	{\mathcal{T}'} \&\&\&\& {\mathcal{P}(\mathcal{S})}
	\arrow[from=1-3, to=2-4]
	\arrow[from=1-3, to=3-1]
	\arrow["E", from=2-4, to=3-5]
\end{tikzcd}\]
\end{remark}

\newpage

\begingroup
\setlength{\emergencystretch}{.5em}
\RaggedRight
\printbibliography
\endgroup

\end{document}